\documentclass[preprint,review,9pt]{elsarticle}
\usepackage{bbm}
\usepackage{amsmath,amssymb,amsthm}
\usepackage[title]{appendix}
\usepackage{color}
\usepackage{xcolor}
\usepackage{graphicx}
\usepackage{enumerate}
\usepackage{ifpdf}
\usepackage{algorithmic}
\usepackage{float}
\usepackage{CJK}
\usepackage{listings}
\usepackage{bbm}
\usepackage{bm}
\usepackage{mathrsfs}
\usepackage{amsmath}
\usepackage{enumerate}
\usepackage{multirow,booktabs}
\usepackage{caption}
\usepackage{makecell}
\usepackage{epstopdf}
\usepackage[percent]{overpic}
\usepackage{graphicx,graphics,epsfig}
\usepackage{mathrsfs}
\usepackage{arydshln}
\usepackage{lipsum}
\usepackage{enumerate}
\usepackage{color}

\usepackage[normalem]{ulem}
\usepackage{soul}
\usepackage{algorithm}

\usepackage{setspace}

\usepackage{tikz}
\usepackage{pgfplots} 

\newtheorem{remark}{Remark}
\usepackage{multirow,booktabs}
\usepackage{makecell}
\usepackage{subfigure}

\makeatletter\@addtoreset{equation}{section} \makeatother
\newtheorem{theorem}{Theorem}[section]
\newtheorem{lemma}{Lemma}[section]

\DeclareMathOperator{\Range}{Range}

\def\R{\mathbb{R}}
\def\H{E}
\def\N{\mathcal{N}}

\def\S{\mathcal{S}}
\def\Sph{\mathbb{S}}

\def\d{\mathrm{d}}

\def\U{\mathcal{U}}
\def\eref#1{{\rm (\ref{#1})}}

\def\qed{~\relax\ifmmode\hskip2em \Box
	\else\unskip\nobreak\hskip1em \hfill$\Box$
	\fi \newline}

\def\lchi{\chi_{\xi}^{\text{loc}}}

\def\dt#1{{\dot{#1}}}

\def\trialsp{{V_{X,\phi_m}}}

\def\IP#1#2{{\left\langle{#1},\,{#2}\right\rangle}}
\def\IPSph#1#2{{{\IP{#1}{#2}}_{L^2(\S)}}}

\def\VarDu{\frac{\delta E}{\delta u}}

\DeclareMathOperator*{\myspan}{span}

\usepackage{geometry}
\begin{document}
	\begin{frontmatter}
		\title{Structure-preserving kernel-based meshless methods for solving dissipative PDEs on surfaces
			\tnoteref{label1}}
		\tnotetext[label1]{The work of the first author was supported by  NSFC (No. 12101310), NSF of Jiangsu Province (No. BK20210315), and the Fundamental Research Funds for the Central Universities (No. 30923010912), The work of the second author was supported by the General Research Fund (GRF No. 12301520, 12301021, 12300922) of Hong Kong Research Grant Council. The third author was supported by NSFC (No. 12001261, 12361086) and NSF of Jiangxi Province (No. 20212BAB211020).}
		
		\author{Zhengjie Sun\fnref{label2}}
		\ead{zhengjiesun@njust.edu.cn}
		\author{Leevan Ling\fnref{label3}}
		\ead{lling@hkbu.edu.hk}
		\author{Meng Chen\fnref{label4}}
		\ead{chenmeng_math@ncu.edu.cn}
		
		\address[label2]{School of Mathematics and Statistics, Nanjing University of Science and Technology, Nanjing, China}
		\address[label3]{Department of Mathematics, Hong Kong Baptist University, Kowloon Tong, Hong Kong}
		\address[label4]{School of Mathematics and Computer Sciences,
			Institute of Mathematics and Interdisciplinary Sciences, Nanchang University, Nanchang, China}



		\begin{abstract}
			In this paper, we propose a general meshless structure-preserving Galerkin method for solving dissipative PDEs on surfaces. By posing the PDE in the variational formulation and simulating the solution in the finite-dimensional approximation space spanned by (local) Lagrange functions generated with positive definite kernels, we obtain a semi-discrete Galerkin equation that inherits the energy dissipation property. The fully-discrete structure-preserving scheme is derived with the average vector field method. We provide a convergence analysis of the proposed method for the Allen-Cahn equation. The numerical experiments also verify the theoretical analysis including the convergence order and structure-preserving properties. Furthermore, we provide numerical evidence demonstrating that the Lagrange function and the coefficients generated by a restricted 
			kernel decay exponentially, even though a comprehensive theory has not yet been developed.
		\end{abstract}
		\begin{keyword}
			Kernel-based method; dissipative PDEs; structure-preserving scheme; Lagrange functions; surface PDEs.\\
			AMS Subject Classifications: 65M60, 65M12, 65D30, 41A30
		\end{keyword}

	\end{frontmatter}

	\section{Introduction}
	\label{Sec:Introduction}
	Partial differential equations (PDEs) exhibit intrinsic properties associated with the underlying geometric structures (e.g., symplecticity, multi-symplecticity, Lie group structure, symmetry, etc.) or fundamental physical principles (e.g., energy conservation, energy dissipation, momentum conservation, positivity, etc.) \cite{feng2010symplectic,hairer2006geometric}.
	Preserving these inherent features in numerical schemes is crucial, as it allows for the design of more precise and stable methods for simulating PDEs.
	Over the past several decades, extensive research has been conducted on structure-preserving numerical schemes \cite{bhatt2017structure,Bridges2,christiansen2011topics,frank2006linear,gong2014some,gong2017structure,mclachlan1993symplectic,sanz1992symplectic}. In this regard, we focus on the energy dissipation property that is particularly significant for solving dissipative PDEs on surfaces like gradient flow models. These models find extensive applications in numerous fields, such as physics and engineering, including but not limited to phase transition, multi-phase fluid flow, crystal growth, image analysis, and others. Therefore, it is of utmost importance to consider the energy dissipation feature during the numerical simulation.

	PDEs on surfaces have captured significant attention in the scientific community owing to their ability to describe numerous phenomena in applied sciences \cite{bertalmio2001variational,chaplain2001spatio,diewald2000anisotropic,dziuk2013finite,greer2006improvement,johnson1991simulation,novak2007diffusion,turk1991generating,valizadeh2019isogeometric}. However, analytically solving such PDEs on surfaces is often unfeasible due to their inherent complexity. Thus, suitable numerical methods must be developed to approximate solutions.
	There are three commonly used numerical formulations: The intrinsic formulation utilizing the intrinsic coordinates of the surface or a triangulated representation of the surface, such as the surface finite element method \cite{dziuk2013finite}. Next, the embedding formulation extends the PDE from the surface to some narrow-band domain containing the surface \cite{macdonald2010implicit}, and lastly, extrinsic formulation processing the surface problem directly on the surface using projection operators to tangent spaces \cite{chen2020extrinsic}.

	To turn these formulations into solution techniques, kernel-based methods present an attractive alternative due to their flexible and efficient meshless framework for solving PDEs on surfaces \cite{chen2020extrinsic,cheung2018kernel,fuselier2013high,kunemund2019high,legia04galerkin,le2006continuous,lehoucq2016meshless,lehto2017radial,mirzaei2018petrov,narcowich2017novel,shankar2015radial,shankar2020robust,wendland2020solving,yan2023kernel}. Their adaptive resolution, interpolation capabilities, high accuracy, ease of handling boundaries, and exemption from explicit mesh generation make them well-suited for a variety of challenging surface PDE problems.
	In \cite{narcowich2017novel}, we can find a novel Galerkin method using localized Lagrange bases that was proposed to solve elliptic PDEs on spheres, demonstrating computational advantages through parallelization. This was later extended to semilinear parabolic \cite{kunemund2019high} and nonlocal \cite{lehoucq2016meshless} PDEs on spheres. While these methods showed success on spherical domains, their applicability is limited and they may neglect inherent PDE properties. To address these limitations, this paper proposes a meshless structure-preserving Galerkin approach for general dissipative PDEs on surfaces, building upon kernel-based techniques while preserving the essential dissipation property.

	

	Throughout this article, $\S\subset \R^d$ is assumed to be a closed, compact, connected, $C^{\infty}$ surface of intrinsic dimension $d_{\S}=d-1$. Consider the function space $\U$ containing
	differentiable functions
	$u: \S\times(0,T] \to \R$.
	We deal with   PDE in the general form \cite{celledoni2012preserving,miyatake2015note,eidnes2018adaptive} of
	\begin{equation}\label{eq:EvolPDE}
		\dot{u}=\N \VarDu[u],
	\end{equation}
	for some  functional $\H:\U\to\R$ with the functional derivative ${\displaystyle\VarDu}$,
	and  linear negative (semi-) definite operator $\N$ , i.e.
	\begin{equation}\label{eq:NegDefiOp}
		\IP {u}{\N u}_{L^2(\S)}\leq 0, ~~\forall u\in \U.
	\end{equation}
	Then, PDE \eqref{eq:EvolPDE} has the following dissipation property,
	\begin{equation}\label{eq:DisspProp}
		\frac{\d \H}{\d t}=\IP{\VarDu}{\dot{u}}_{L^2(\S)}=\IP{\VarDu}{\N\VarDu}_{L^2(\S)}\leq 0.
	\end{equation}
	
	Typical examples for energy dissipative PDE \eqref{eq:EvolPDE} include the Allen-Cahn equation \eqref{eq:ACeq} and the Cahn-Hilliard equation \eqref{eq:CHeq}.
	In the literature, there exist several systematic approaches to enforce invariant properties on numerical schemes for solving dissipative PDEs \eqref{eq:EvolPDE} (mostly) on regular domains with structured nodes. In \cite{celledoni2012preserving}, the authors proposed a general dissipative method to solve \eqref{eq:EvolPDE} using the finite difference method. Later, \cite{eidnes2018adaptive} developed a partition of unity method in the Galerkin formulation with adaptive energy-conserving properties, which can also be applied to dissipative systems \cite{sun2021novelenergy}.

	The objective of this paper is to generalize our previously proposed energy-preserving meshless Galerkin method \cite{sun2022kernel} to obtain a novel meshless Galerkin method that preserves the dissipative structure of PDEs on surfaces.
	The remaining sections of this paper are organized as follows. Section \ref{sec:Prelim} provides preliminaries for the proposed meshless method, in particular, the local Lagrange functions. Section \ref{sec:MeshlessStructurePres} establishes the framework of the meshless structure-preserving Galerkin method.
	The applications of the meshless Galerkin schemes for solving the Allen-Cahn equation and the Cahn-Hilliard equation are discussed in Section \ref{sec:MeshlessEnerConser}. The convergence analysis is presented in the following section. Section \ref{sec:KenelQuad} also delves into the details of the numerical discretization for the mass matrix and stiffness matrix. Subsequently, in Section \ref{sec:NumerExp}, we report some numerical results for the Allen-Cahn equation and the Cahn-Hilliard equation on surfaces, highlighting the performance of both intrinsic and restricted kernels. Finally, the paper concludes with a summary in the last section.

	\section{Preliminaries}\label{sec:Prelim} In this section, we will briefly introduce some key ingredients in constructing our meshless Galerkin method via positive definite kernels. Let $X\subset\S$ be a set of $N_X$ distinct centers on the surface. The distance on the surface $\text{dist}(\xi,\eta)$ is the length of the shortest geodesic joining $\xi$ and $\eta$.
	The three quantities associated to the set $X$, the fill distance, separation distance, and mesh ratio are defined respectively by
	\begin{equation}\label{eq:filldistance}
		h_X:=\sup_{x\in\S}\text{dist}(x,X),~~q_X:=\frac{1}{2}\inf_{\substack{\xi,\eta\in X\\\xi\neq\eta}}\text{dist}(\xi,\eta),~~\rho_{X}:=\frac{h_X}{q_X}.
	\end{equation}
	We work with \emph{positive definite} kernels $\Phi(\xi,\eta):\S\times\S\rightarrow\R$ that reproduce Sobolev spaces in $\S$.
	One important example is the Mat\'{e}rn kernel, denoted by $\phi_m$, which is the reproducing kernel for the Sobolev space $W_2^m(\S)$. This kernel was introduced and thoroughly studied in \cite{hangelbroek-2010SIAMMA-kernel,hangelbroek2018direct}. Notably, the kernel is intrinsically defined and serves as the fundamental solution of the elliptic differential operator, $\mathcal{L}=\sum_{j=0}^{m}(\nabla_S^j)^*\nabla_S^j$ of order $2m$, where $\nabla_S$ is the covariant derivative defined on $\S$, and $(\nabla_S^j)^*$ is the adjoint of $\nabla_S^j$ with respect to the $L_2(\S)$ (see \cite{aubin2012nonlinear,hangelbroek-2010SIAMMA-kernel} for details). The theory can also be extended to conditionally positive definite kernels, extensively studied in \cite{kunemund2019high,narcowich2017novel}.
	
	

	Now, for a finite set $X\subset\S$, we define a finite-dimensional approximation space on $\S$:
	\begin{equation}\label{eq:FiniteDimSpace}
		V_X=V_{X,\phi_m}=\myspan_{x\in X  {\subset \S}}\big\{\phi_m(\,\cdot\,,x)\big\}
		{\subset W_2^{m}(\S)}.
	\end{equation}
	Here $\phi_m(\,\cdot\,,x):=\phi_m(\text{dist}(\cdot,x))$. Suppose $f:\S\rightarrow\R$ is the unknown function and $f|_X$ is the vector with nodal values of $f$ at  $X$. Then, there exists a unique interpolant $I_Xf$  {of $f\in W_2^{m}(\S)$ on $X$} from the space $V_X$ in the form of
	\begin{equation}
		I_Xf = \sum_{x\in X}a_{\xi}\phi_m(\,\cdot\,,x).
	\end{equation}
	The coefficients $a_{\xi}\in \R$ are uniquely determined by the interpolation condition $I_Xf(\eta)=f(\eta)$ for all $\eta\in X$.

	\subsection{Lagrange functions} In this section, we only consider the  Mat\'{e}rn kernel $\phi_m$. Instead of using
	the Mat\'{e}rn kernel in the standard form to span the approximation space,  $V_{X,\phi_m}\subset W_2^{m}(\S)$ can also be spanned
	using the set of \emph{Lagrange functions}, a.k.a. \emph{cardinal functions}, denoted by $\chi_{\xi}\in V_{X,\phi_m}$ for all $\xi \in X\subset \S$.
	The Lagrange function centered at any $\xi\in X$ is the unique interpolant $\chi_{\xi}\in V_{X,\phi_m}$ satisfying
	$\chi_{\xi}(\xi)=1$ and $\chi_{\xi}(\eta)=0$ for all $\eta\in X\setminus \xi$.
	Since $\chi_{\xi}\in V_{X,\phi_m}$, the Lagrange function has the representation
	\begin{equation}\label{eq:LagrangFunc}
		\chi_{\xi}(\cdot)=\sum_{\eta\in X}\alpha_{\xi,\eta}\phi_m(\cdot,\eta),
	\end{equation}
	for some coefficients $\alpha_{\xi,\eta}\in\R$.
	Consequently, the
	interpolant of $f$
	using the Lagrange functions as a basis function, which is defined by nodal values of $f$ on $X$, i.e.,
	$$I_{X,\phi_m}f=\sum_{\xi\in X}
	\alpha_\xi \chi_{\xi} \text{\quad with } \alpha_\xi=f(\xi).$$
	From \cite{hangelbroek2018direct}, we know that there exists constants $C$, $\nu$,  such that
	\begin{align}
		|\alpha_{\xi,\eta}|&\leq C q_X^{d-2 {s}}\text{exp}\Big(-\nu\frac{\text{dist}(\xi,\eta)}{h_X}\Big),
		\\
		|\chi_{\xi}(x)|&\leq C\rho_X^{ {s}-d/2}\text{exp}\Big(-2\nu\frac{\text{dist}(x,\xi)}{h_X}\Big).
	\end{align}
	
	We note that both the Lagrange functions (associated with the Mat\'{e}rn kernel) and its coefficients decay exponentially. Thus, these Lagrange functions  can be effectively approximated  only by using a few nearby kernels
	\cite{fuselier2013localized,kunemund2019high,lehoucq2016meshless,narcowich2017novel}. 
	
	\subsection{Local Lagrange functions}  Establishing the Lagrange basis for $V_{X,\phi_m}$ requires solving a large system of equations if the number of centers becomes large.
	To reduce the computation and allow parallelism, Fuselier et al. \cite{fuselier2013localized} introduced the \emph{local} Lagrange functions of surface splines that are conditionally positive definite for $V_{X,\phi_m}$, which are obtained by solving relatively small linear systems
	associated with some local stencils.
	Here, we present the positive definite case of Mat\'{e}rn kernel.

	Let $\xi\in X$ be a trial center and $X_{\xi}$ be its local neighborhood of $\xi$ containing nearest centers:
	\begin{equation}\label{eq:LocalSet}
		X_{\xi}:=\{\eta\in X:\text{dist}(\xi,\eta)\leq r_{X}\},
	\end{equation}
	within radius $r_{X}:=Kh_X|\log h_X|$ for some parameter $K$ \cite{fuselier2013localized}. The local Lagrange function is the interpolant to the reduced cardinal conditions
	$\lchi(\xi)=1$ and $\lchi(\eta)=0$ for all $\eta\in X_\xi\setminus \xi$
	in the form of
	\begin{equation}\label{eq:ConstrLocalLag}
		\lchi(\cdot)=\sum_{\eta\in X_{\xi}}\alpha_{\eta,\xi}^{\text{loc}}\phi(\cdot,\xi).
	\end{equation}
	Based on numerical results in \cite{fuselier2013localized,kunemund2019high} and confirmed with our numerical results in Section~\ref{sec:NumerExp}, using $K=7$ is sufficient to obtain accuracy of the global counterpart.

	\subsection{Projection operator} We define an $L^2$-orthogonal projection operator 
	$\mathcal{P}: L^2(\S) \rightarrow  V_{X,\phi_m}$
	for any $f\in L^2(\S)$ by
	\begin{equation}\label{eq:ProjOper}
		\IPSph {\mathcal{P}f}{v}=\IPSph{f}{v},~~\forall v\in V_{X,\phi_m}.
	\end{equation}
	We have established  convergence estimates for $\|\mathcal{P}f-f\|$ in the Euclidean space \cite[Lem 2.2]{sun2022kernel}. For  closed and smooth surfaces, we have the following:
	\begin{lemma}\label{thm:ProjOper}
		Let $I$ be a bounded interval of $\mathbb{R}$. Assume that a  smooth  function $f:I\rightarrow\mathbb{R}$ satisfies $f(0)=0$. For any trial function $\chi\in V_{X,\phi_m}$ with $\Range(\chi)\subseteq I$, it holds that
		$$\|\mathcal{P}f(\chi)-f(\chi)\|_{L^2(\S)}\leq C h_X^{m}(1+\|\chi\|_{L^{\infty}(\S)})^m\|\chi\|_{H^m(\S)}.
		$$
	\end{lemma}
	\begin{proof}
		Our previous proof in \cite[Lem 2.2]{sun2022kernel} for flat geometry also works on $\S$ by substituting  $f(\chi)\in H^m(\S)$ in places of $F'(\chi)$ in the original proof.
	\end{proof}
	
	\section{Kernel-based structure-preserving Galerkin method for dissipative PDEs}
	\label{sec:MeshlessStructurePres}
	Using the Lagrange basis functions,
	Narcowich et al. \cite{narcowich2017novel} developed a novel Galerkin method for solving elliptic equations on the sphere. Afterward, this method was applied to solve semilinear parabolic equations \cite{kunemund2019high} and non-local diffusion equations \cite{lehoucq2016meshless}. However, these methods all focused on the sphere by using restricted surface splines and cannot easily applied to other surfaces. Besides, the inherent properties of time-dependent PDEs are not considered. Therefore, in this section, we will utilize the novel meshless Galerkin method to construct structure-preserving schemes for dissipative PDEs with inherent physical properties \eqref{eq:EvolPDE} defined on general surfaces.
	
	\subsection{Spatial discretization} We will follow the main idea from \cite{eidnes2018adaptive} but with the meshless Galerkin method and on surfaces. We seek the weak solution to \eqref{eq:EvolPDE} by solving
	\begin{equation}\label{eq:WeakSolu}
		\IPSph{\dot{u}}{v}=\IPSph{\N\VarDu}{v},~~\forall v\in\U.
	\end{equation}
	Be reminded that $\N$ is a self-adjoint negative (semi-)definite operator. By approximating the unknown function $u\in\U$ with $u_h \in \trialsp$ in the finite-dimensional space with time-dependent coefficients, i.e.,
	\begin{equation}\label{alphadef}
		u(t,\cdot) \approx u_h(t,\cdot) = \sum_{\xi\in X}\alpha_{\xi}(t)\chi_{\xi}(\cdot),
	\end{equation}
	we arrive at the following Galerkin equation
	\begin{equation}\label{stdGalerkin}
		\IPSph{\dot{u}_h}{v_h}=\IPSph{\N\VarDu[u_h]}{v_h}=\IPSph{\VarDu[u_h]}{\N v_h},~~\forall v_h\in\trialsp,
	\end{equation}
	which does not preserve the dissipation structure.
	
	To obtain a dissipative scheme, we project the nonlinear function $\VarDu[u_h]$ onto the approximation function space $\trialsp$ in the sense of \eqref{eq:ProjOper}, i.e.
	\begin{equation}\label{PdEuh}
		\mathcal{P}\Big(\VarDu[u_h]\Big)(t,\cdot):=\sum_{\xi\in X}\lambda_{\H,\xi}(t)\chi_{\xi}(\cdot).
	\end{equation}
	The coefficients $\lambda_{\H,\xi}$ are determined by using definition \eref{eq:ProjOper} of the projection and the variational derivative. For all  $\chi_{\eta}\in\trialsp$, we want the followings hold:
	\begin{equation*}
		\begin{aligned}
			&
			\left\langle\sum_{\xi\in X}\lambda_{\H,\xi}(t)\chi_{\xi},\chi_{\eta}\right\rangle=
			\big\langle   \VarDu[u_h],\chi_{\eta}\big\rangle=
			(\nabla_{\alpha}\H[u_h])^T\chi_{\eta}.
		\end{aligned}
	\end{equation*}
	where $\nabla_{\alpha}$ represents the gradient with respect to the coefficients $\alpha$.
	Therefore, the coefficients in \eref{PdEuh} is given by
	\begin{equation}\label{eq:MassMatrix}
		\lambda_{\H}= \bigg( \underbrace{\Big[\IPSph{\chi_{\xi}}{\chi_{\eta}}\Big]_{\xi,\eta\in X} }_{=:A\in\R^{N_X\times N_X}} \bigg)
		^{-1}(\nabla_{\alpha}\H[u_h]):= A^{-1}  (\nabla_{\alpha}\H[u_h]) .
	\end{equation}
	Replacing the variational derivative $\displaystyle\VarDu[u_h]$ by its projection in the Galerkin equation \eref{stdGalerkin} yields a new equation:
	\begin{equation}\label{newGalerkin}
		\sum_{\xi\in X}\dot{\alpha}_{\xi}\IPSph{\chi_{\xi}}{\chi_{\eta}}=\sum_{\xi\in X}\lambda_{\H,\xi}\IPSph{\chi_{\xi}}{\N \chi_{\eta}},~~\forall \chi_{\eta}\in\trialsp\subset\U.
	\end{equation}
	By letting
	$D=\Big(\IPSph{\chi_{\xi}}{\N\chi_{\eta}}\Big)$,
	we have the semi-discrete ODE system
	\begin{equation}\label{eq:EvolPDESemiDis}
		\dot{\alpha} =A^{-1}D\lambda_{\H}=A^{-1}DA^{-1}(\nabla_{\alpha}\H[u_h]).
	\end{equation}
	We emphasize that the matrices $A$ and $D$ are both symmetric. Moreover, the semi-discrete ODE system \eqref{eq:EvolPDESemiDis} has a dissipative energy $\H[u_h]$.
	
	\begin{theorem}\label{thm:DissipationPropSemiDisc}
		The semi-discrete system \eqref{eq:EvolPDESemiDis} satisfies the dissipation property, $\frac{d\H[u_h]}{dt}\leq 0$.
	\end{theorem}
	\begin{proof}
		Since $\N$ is a self-adjoint negative (semi-)definite linear operator,
		we can easily obtain that $D$ is a negative (semi-)definite matrix and
		\begin{equation}
			\begin{aligned}
				\frac{d\H[u_h]}{dt}=&\IPSph{\nabla_{\alpha}\H[u_h]}{\dot{\alpha}}=\IPSph{\nabla_{\alpha}\H[u_h]}{A^{-1}DA^{-1}\nabla_{\alpha}\H[u_h]}\\
				=&\IPSph{A^{-1}\nabla_{\alpha}\H[u_h]}{DA^{-1}\nabla_{\alpha}\H[u_h]}\leq 0.
			\end{aligned}
		\end{equation}
	\end{proof}
	
	

	\subsection{Time discretization}
	We now have a meshless Galerkin equation \eqref{newGalerkin} that preserves the dissipative structure of PDEs in the semi-discrete (spatial) discretization. To obtain a fully discrete scheme, it remains to introduce an appropriate time integrator that possess structure-preserving properties for dissipative systems. In particular, the average vector field (AVF) method \cite{quispel2008new} has been shown to automatically preserve important invariants like the energy \cite{cai2016numerical,celledoni2010energy,celledoni2014minimal,cieslinski2014improving}  and the dissipative property \cite{celledoni2012preserving} for ODE and PDE systems.
	
	Consider a first-order ODE system in the general form of
	$$\frac{dy}{dt}=f(y),$$
	for some Lipschitz continuous function $f:\R\to\R$.
	We partition the time interval $[0,T]$ into $N_t$ sub-intervals with the time step $\triangle t=T/N_t$. Let $y^n$ be the approximation of $y(t_n)$ at $t_n=n\triangle t$. Then,  the second-order AVF scheme is given by
	\begin{equation}\label{AVF}
		\frac{y^{n+1}-y^n}{\triangle t}=\int_{0}^1f((1-\xi)y^n+\xi y^{n+1})\d\xi.
	\end{equation}
	Now, applying AVF method to the semi-discrete system \eqref{eq:EvolPDESemiDis} and let $u_h^n$ be the approximation of $u^n$ with $\alpha^n=\alpha(t_n)$, we obtain a full-discrete dissipative scheme
	\begin{equation}\label{eq:FullDisScheme}
		\frac{\alpha^{n+1}-\alpha^n}{\triangle t} = A^{-1}DA^{-1}\int_{0}^1 \nabla_{\alpha}\H[(1-\xi)u_h^n+\xi u_h^{n+1}]\d\xi,
	\end{equation}
	to the equation \eqref{eq:EvolPDE}.
	
	Note in \eqref{AVF} and \eqref{eq:FullDisScheme} that the vector field function takes the form of  $f(y)=\overline{\N}\nabla _y\overline{\H}(y)$ with a negative-definite matrix $\overline{\N}$.
	This implies $\overline{\H}$ is a Lyapunov function of this ODE and the dissipation/decay of $\overline{\H}(y)$ over time. The AVF method is designed to preserve this dissipative structure \cite{celledoni2010energy,celledoni2012preserving}, i.e.
	$$\overline{\H}(y^{n+1})\leq \overline{\H}(y^n).$$
	Up to now, we have presented a general framework of kernel-based structure-preserving schemes for dissipative PDEs. In the following, we will explain our algorithms through specific equations like the Allen-Cahn and the Cahn-Hilliard equation.

	\section{Fully-discrete, and dissipative schemes for the Allen-Cahn and the Cahn-Hilliard equation}
	\label{sec:MeshlessEnerConser}
	
	It is well-known that the Allen-Cahn equation can be formulated as an $L^2$ gradient flow \cite{dziuk2013finite,xiao2017stabilized}, where the time derivative of the solution is the $L^2$ gradient of a Lyapunov energy functional. The Cahn-Hilliard equation can be formulated analogously as an $H^{-1}$ gradient flow.
	Their Lyapunov functional, which measures the free energy driving the phase separation/domain growth process, is defined as follows:
	\begin{equation}\label{eq:ACeq_energy}
		\H[u;G]=\int_{\S}\left[\frac{\varepsilon^2}{2}|\nabla_{\S} u|^2+G(u)\right]\d \mu,
	\end{equation}
	with the Ginzburg-Landau free energy density $G(u)=\frac{1}{4}(u^2-1)^2$ and the interface width $\varepsilon$.
	
	By employing the energy function $\H$ in \eqref{eq:ACeq_energy} and taking $\N=-I$, the PDE \eqref{eq:EvolPDE} becomes  the Allen-Cahn equation
	\begin{equation}\label{eq:ACeq}
		\dt{u}(x,t)-\varepsilon^2\Delta_{\S} u(x,t)+G'\big(u(x,t)\big)=0, ~~~(x,t)\in \S\times(0,T],
	\end{equation}
	and the counterpart to the Cahn-Hilliard equation  with $\N=\Delta_{\S}$
	is given as
	\begin{equation}\label{eq:CHeq}
		\dt{u}(x,t)-\Delta_{\S}(-\varepsilon^2\Delta_{\S} u(x,t)+G'(u(x,t))=0, ~~~(x,t)\in \S\times(0,T],
	\end{equation}
	with a Laplace-Beltrami operator $\Delta_{\S}$ defined on $\S$ and a nonlinear term described by a smooth real-valued function $G':\mathbb{R}\rightarrow\mathbb{R}$.
	One notable characteristic of gradient flow models is the dissipative nature of the total energy, i.e.,
	$\frac{d}{dt} E[u;G]\leq 0$.
	
	Both \eqref{eq:ACeq} and \eqref{eq:CHeq} are closely associated with the Lyapunov functional. In accordance with the methodology proposed in Section \ref{sec:MeshlessStructurePres}, our primary objective is to discretize the energy functional \eqref{eq:ACeq_energy}.  By employing the kernel-based numerical solution
	$u_h=\displaystyle\sum_{\xi\in X}\alpha_{\xi}\chi_{\xi}$
	in terms of Lagrange basis functions in the finite-dimensional space $V_{X,\phi_m}$, we obtain
	\begin{equation}\label{eq:ACSemiDisEqEnergy}
		\H_d[\alpha;G] = \frac{\varepsilon^2}{2}\alpha^T
		\underbrace{\Big[\IPSph{\nabla_{\S} \chi_{\xi}}{\nabla_{\S}\chi_{\eta}}\Big]_{\xi,\eta\in X}}_{=:B \in \R^{N_X\times N_X}}
		\alpha + \IPSph{G(u_h)}{1}.
	\end{equation}
	%
	Here, $E_d$ represents the discrete version of the continuous energy $E$. For the sake of evaluating \eref{eq:FullDisScheme}, we explicitly write  the gradient of $\H_d[\alpha;G]$ with respect to the vector of nodal solution values $\alpha = [\alpha_\eta]_{\eta\in X} \in \R^{N_X}$ as
	$$\nabla_{\alpha} \H_d[\alpha;G] = \varepsilon^2B\alpha+
	\underbrace{\Big[\IPSph{ {G'}(u_h)}{\chi_{\eta}}\Big]_{\eta\in X}}_{=:\lambda_{G} \in \R^{N_X}},$$
	where $\lambda_{G}$ is a vector of length $N_X$.
	{The essential step in deriving dissipative schemes involves incorporating the semi-discrete (in space) energy, as defined in Equation \eqref{eq:ACSemiDisEqEnergy}, into the fully-discrete scheme \eqref{eq:FullDisScheme}. By employing the technique we developed in \cite{sun2022kernel}, we rewrite the integral in \eqref{eq:FullDisScheme} as follows:}
	\begin{equation}\label{int discret energy}
		\begin{aligned}
			&\int_{0}^1 \nabla_{\alpha}E_d[(1-\xi)u_h^n+\xi u_h^{n+1}]\d\xi\\
			~&=\int_{0}^1 \Big[\varepsilon^2 B\Big((1-\xi)\alpha^n+\xi \alpha^{n+1}\Big)+ {\Big(\IPSph{ {G'}((1-\xi)u_h^n+\xi u_h^{n+1})}{\chi_{\eta}}\Big)_{\eta\in X}} \Big]\d\xi\\
			~&=\varepsilon^2 B\frac{\alpha^n+\alpha^{n+1}}{2}
			+ \underbrace{\left[ \IPSph{\frac{G(u_h^{n+1})-G(u_h^n)}{u_h^{n+1}-u_h^n}}
				{\chi_{\eta}}\right]_{\eta\in X} }_{=:\tilde{\lambda}_{ {G}}^{n+1} \in \R^{N_X}}.
		\end{aligned}
	\end{equation}
	
	With the above discretization set up and by using the definition of $D$ and $\N=-I$, we have
	\begin{equation}
		D=\Big[\IPSph{\chi_{\xi}}{\N\chi_{\eta}}\Big]_{\xi,\eta\in X}=-A,
	\end{equation}
	and thus obtain the following fully-discrete  scheme for the Allen-Cahn equation  in \eqref{eq:FullDisScheme}:
	\begin{equation}\label{eq:ACfullDiscretization}
		\frac{\alpha^{n+1}-\alpha^n}{\triangle t}=-\varepsilon^2 A^{-1}B\frac{\alpha^{n}+\alpha^{n+1}}{2}-A^{-1}\tilde{\lambda}_{ {G}}^{n+1}.
	\end{equation}
	Similarly, the operator $\N=\Delta_{\S}$ leads to
	\begin{equation}
		D=\Big[\IPSph{\chi_{\xi}}{\N\chi_{\eta}}\Big]_{\xi,\eta\in X}=-B,
	\end{equation}
	and the fully-discrete scheme for the Cahn-Hilliard equation is given by:
	\begin{equation}\label{eq:CHfullDiscretization}
		\frac{\alpha^{n+1}-\alpha^n}{\triangle t}=-\varepsilon^2 A^{-1}BA^{-1}B\frac{\alpha^{n}+\alpha^{n+1}}{2}-A^{-1}BA^{-1}\tilde{\lambda}_{ {G}}^{n+1}.
	\end{equation}
	
	The following properties can be proved by using  \ref{thm:DissipationPropSemiDisc} and the results in \cite[Eqn. 23]{celledoni2012preserving}, thus we omit the proof here.
	\begin{theorem}\label{thm:FullDiscrete}
		The  {fully-discrete}  energy of the schemes \eqref{eq:ACfullDiscretization} and \eqref{eq:CHfullDiscretization} is non-increasing. Let $\H_{d,G}^n = \frac{\varepsilon^2}{2}(\alpha^n)^T B \alpha^n + \IPSph{G(u_h^n)}{1}$. Then, we have
		\begin{equation}\label{eq:ACFullDisEqEnergy}
			\H_{d,G}^n\leq \H_{d,G}^{n-1}, 
		\end{equation}
		for all $1\leq n\leq N_t$.
	\end{theorem}

	\subsection{ {Error estimate for the Allen-Cahn equation}}
	The convergence analysis of the semilinear wave equation on a bounded domain in $\R^{d}$ for spatial discretization was previously established in \cite{sun2022kernel}, focusing on compactly supported solutions to circumvent boundary conditions. By employing a similar approach, we can extend the existing convergence results to the Allen-Cahn equation on closed, smooth surfaces.
	The estimate of the Cahn-Hilliard equation can be handled similarly but needs more regularities on the kernel $\phi_{m}$.
	The error associated with semi-discretization is detailed in the subsequent theorem.
	\begin{theorem}\label{thm:SpatialErr}
		Let $u\in C([0,T];H^{\sigma}(\S))$ be the exact solution to the Allen-Cahn equation \eqref{eq:ACeq}. Suppose that $f:\R\to\R$ is Lipschitz continuous on a domain $\tilde{I}$ that contains $\text{Range}(u)$ as a proper subset. Let $\phi_m$ be the Mat\'{e}rn kernel and $\trialsp$ be the associated approximation space \eqref{eq:FiniteDimSpace} with the fill distance $h_X=h$. Let $u_h(t)=\sum_{\xi}\alpha_{\xi}(t)\chi_{\xi}\in \trialsp$ and $\alpha$ be the solution of equation \eqref{eq:EvolPDESemiDis} for the Allen-Cahn equation.
		If the initial conditions are approximated with the following accuracy,
		\begin{equation}\label{eq:ACInitErr}
			\|u(0)-u_h(0)\|_{H^j(\S)}\leq C h^{\sigma-j},~~j=0,1,
		\end{equation}
		then for $m\geq \sigma >d_{\S}/2$, the semi-discrete approximation solution $u_h(t)$
		in the form of \eqref{alphadef} with coefficients by \eqref{eq:ACfullDiscretization}
		converges to $u(t)$ with the error bound
		\begin{equation}\label{eq:ACSpatErr}
			\|u(t)-u_h(t)\|_{H^j(\S)}\leq Ch^{\sigma-j}\|u\|_{H^{\sigma}(\S)},~~j=0,1.
		\end{equation}
	\end{theorem}
	\begin{proof}
		The proof can be completed by combining results in  \cite{hangelbroek2018direct,kunemund2019high,sun2022kernel}. Using the estimates in \cite{hangelbroek2018direct}, we can verify that the finite dimensional space constructed with the Mat\'{e}rn kernel satisfies three properties (inverse estimate, simultaneous approximation and convergence order) in \cite[Definition 2.2]{kunemund2019high} for closed, compact, connected and smooth surfaces. Thus, for the Galerkin equation
		\begin{equation}\label{eq:ACGalerkinEq}
			\IPSph{\partial_t u_h}{\chi_{\eta}}+\varepsilon^2\IPSph{\nabla_{\S} u_h}{\nabla_{\S}\chi_{\eta}}=\IPSph{G'(u_h)}{\chi_{\eta}},~\forall \chi_{\eta}\in V_{X,\phi_m},
		\end{equation}
		with initial conditions satisfying \eqref{eq:ACInitErr}, we   could obtain the error estimate $$\|u(t)-u_h(t)\|_{H^j(\S)}\leq Ch^{\sigma-j}\|u\|_{H^{\sigma}(\S)},$$ similar to \cite[Thm. 2.3]{kunemund2019high}.
		From \cite[Thm. 3.1]{sun2022kernel}, we can verify that the projected semi-discrete equation \eqref{eq:EvolPDESemiDis} is equivalent to \eqref{eq:ACGalerkinEq} and the proof is completed.
	\end{proof}
	
	
	From \cite[Thm 3.15]{kunemund2019high}, we learn that the Crank-Nicolson scheme of the Allen-Cahn equation
	$$
	A\frac{\alpha^{n+1}-\alpha^n}{\triangle t}=-\varepsilon^2 B\frac{\alpha^{n}+\alpha^{n+1}}{2}-g\Big(\frac{\alpha^n+
		\alpha^{n+1}}{2}\Big),
	\text{\quad  {with} }
	g(\alpha)=\bigg[ \Big\langle G'\big(\sum_{\xi\in X}\alpha_{\xi}\chi_{\xi}  \big), \chi_{\eta}\Big\rangle\bigg]_{\eta\in X},
	$$
	provides a second order of convergence in time. The only difference between the Crank-Nicolson scheme and the AVF  {fully-discrete}  scheme \eqref{eq:ACfullDiscretization} is the nonlinear term. To proceed, we prove the following useful lemma about the estimate of the nonlinear function.
	\begin{lemma}\label{lem:NonlinearTerm}
		Let $G'$ be a smooth and Lipschitz continuous function satisfying the conditions in \ref{thm:ProjOper}. Let $u\in C^2([0,T];H^{\sigma}(\S))$ be the exact solution to \eqref{eq:ACeq} with $\Range(u)\subset I$, $u_h^n\in\trialsp$ with $\alpha^n$ be the solution to the Allen-Cahn equation \eqref{eq:ACfullDiscretization} and $e^n=u_h^n-u^n$. Then,
		{there exists some constant $C$ independent of $u$ such that the following estimate holds}
		$$\Big|G'\Big(\frac{u^{n+1}+u^n}{2}\Big)-\frac{G(u_h^{n+1})-G(u_h^n)}{u_h^{n+1}-u_h^n}\Big|\leq C \Big(|e^{n+\frac{1}{2}}|+|e^n|^2+|e^{n+1}|^2+ {(\triangle t )^2}\Big).$$
	\end{lemma}
	\begin{proof}
		Using the Taylor's expansion at $u_h^{n+\frac{1}{2}}=\frac{1}{2}( u_h^n+u_h^{n+1} )$ to yield
		\begin{equation*}
			\frac{G(u_h^{n+1})-G(u_h^n)}{u_h^{n+1}-u_h^n}=G'(u_h^{n+\frac{1}{2}})+\frac{1}{24}G'''(\eta)(u_h^{n+1}-u_h^n)^2,
		\end{equation*}
		where $\eta$ is between $u_h^n$ and $u_h^{n+1}$. Then there exists a general constant $C$ such that
		\begin{equation*}
			\begin{aligned}
				&\Big|G'(u^{n+\frac{1}{2}})-\frac{G(u_h^{n+1})-G(u_h^n)}{u_h^{n+1}-u_h^n}\Big|\\
				\leq &|G'(u^{n+\frac{1}{2}})-G'(u_h^{n+\frac{1}{2}})|+C|u_h^{n+1}-u_h^n|^2\\
				\leq & C|u^{n+\frac{1}{2}}-u_h^{n+\frac{1}{2}}|+C|u_h^{n+1}-u^{n+1}+u^{n+1}-u^n+u^n-u_h^{n}|^2\\
				\leq &C (|e^{n+\frac{1}{2}}|+|e^n|^2+|e^{n+1}|^2)+C(\triangle t )^2,
			\end{aligned}
		\end{equation*}
		where we have used the Lipschitz condition of $G'$.
	\end{proof}

	
	Now we give the full discretization error estimate.
	\begin{theorem}
		Suppose that $u\in C^2(0,T;H^{\sigma}(\S))$. Let $u_h^n\in\trialsp$ be the solution of \eqref{eq:ACfullDiscretization}.
		Then under the assumptions of Theorem \ref{thm:SpatialErr}, we have the estimate
		\begin{equation}\label{eq:FullDiscretEstimate}
			\|u_h^n-u(t^n)\|_{L_2(\S)}\leq C(\triangle t^2+h^{\sigma}),
		\end{equation}
		where $C$ is independent of  {$u$}, $h$ and $\triangle t$.
	\end{theorem}
	\begin{proof}
		The convergence of Galerkin methods for solving the Allen-Cahn equation had
		been investigated by many authors, see \cite{thomee2006Galerkin,kunemund2019high}.
		Combining Lemma \ref{lem:NonlinearTerm} and following the standard results in classical Galerkin theory \cite[Thm. 13.4]{thomee2006Galerkin} complete the proof.
	\end{proof}
	
	%
	%

	\section{Computable scheme via kernel-based quadrature and local approximation}
	\label{sec:KenelQuad}
	
	
	To compute the fully discrete schemes (\ref{eq:ACfullDiscretization}) and (\ref{eq:CHfullDiscretization}) for the Allen-Cahn and the Cahn-Hilliard equation,  appropriate quadrature is needed to evaluate the surface integrals that arise.
	Kernel-based quadrature formulas for computing surface integrals on spheres and other homogeneous manifolds can be found in early work \cite{fuselier2014kernel}, which was combined with Galerkin methods to solve problems on spheres. However, it could not be easily extended to surface integrals over general surfaces. Later, Reeger \cite{reeger2016numerical} developed a high-order kernel-based quadrature specifically for computing surface integrals on general surfaces.
	We adopt the quadrature rule of Reeger to enable approximation of the surface integrals arising in our numerical schemes.
	For some given set of quadrature points $Y=\{\zeta_1,\ldots,\zeta_{N_Y}\}\subseteq\S$, the quadrature rule of Reeger for any integrable function ${f}$ on $\S$
	is in the form of:
	\begin{equation}\label{eq:KernelQuadrature}
		\int_{\S}{f}(x)\d \mu\approx  \sum_{\zeta\in Y}\omega_{\zeta}{f}(\zeta)=: Q_Y({f}),
	\end{equation}
	with quadrature coefficients $\{ \omega_{\zeta} \} {\subset \R^{N_Y}}$ computed by projection to planar domains, as described in detail in \cite{reeger2016numerical}.

	To approximate the mass matrix $A$ and stiffness matrix $B$ in \eqref{eq:MassMatrix} and \eqref{eq:ACSemiDisEqEnergy}
	in the Galerkin equations, the {two} cases we care involve:
	\begin{equation}\label{eq:InnerProductDiscretization}
		f \in \Big\{ \chi_{\xi}\chi_{\eta}, \quad \nabla_{\S}\chi_{\xi}\cdot\nabla_{\S}\chi_{\eta}
		\Big\},
	\end{equation}
	for all Lagrange functions $\chi_{\xi}$
	for $\xi\in X$ as in \eqref{eq:LagrangFunc} associated with trial centers in a set $X$.

	Denote $A^Y=\big(Q_Y(\chi_{\xi}\chi_{\eta})\big)_{\xi,\eta\in X}$ and $B^Y=\left(Q_Y(\nabla_{\S}\chi_{\xi}\cdot\nabla_{\S}\chi_{\eta})\right)_{\xi,\eta\in X}$ to be the
	$N_X \times N_X$ approximated mass and stiffness matrices obtained by the quadrature nodes $Y$.
	Also let $\tilde{\lambda}_g^Y=(\tilde{\lambda}_{g,\eta}^Y)_{\eta\in X}$ be the approximate vector of length $N_X$ with entries
	\begin{equation}
		\tilde{\lambda}_{g,\eta}^Y=\sum_{\zeta\in Y}\Big(\frac{G(u_h^{n+1})-G(u_h^n)}{u_h^{n+1}-u_h^n}\Big)
		\Big|_{\zeta}\chi_{\eta}(\zeta)\omega_{\zeta}.
	\end{equation}
	We aim to obtain  {an approximation to the} discretized system  {in the form of:}
	\begin{equation}\label{eq:NLWfullDiscretization_local}
		A^Y\frac{\alpha^{n+1}-\alpha^n}{\triangle t}=-B^Y\frac{\alpha^{n}+\alpha^{n+1}}{2}
		-\tilde{\lambda}_g^Y.
	\end{equation}
	The computation of $A^Y$ and $\tilde{\lambda}_g^Y$ is straightforward and involves only nodal values of the Lagrange functions. To compute the entries of $B^Y$, we need the surface gradient. For any surface point $\zeta\in Y$, we denote its normal vector as $\mathbf{n}_{\zeta}$, then the surface gradient of $\phi_m(\cdot,\eta)$ at $\zeta$ can be computed by
	$$\nabla_{\S}\phi_m(\cdot,\eta)|_{\zeta}=(\mathbf{I}-\mathbf{n}_{\zeta}\mathbf{n}_{\zeta}^T)\nabla\phi_m(\zeta,\eta),$$
	where $\nabla$ is the gradient in $\R^3$ with respect to the first variable. The computation of the stiffness matrix reduces to
	\begin{equation}\label{eq:ComputStiffMatrix}
		\begin{aligned}
			Q_Y(\nabla_{\S}\chi_{\xi}{\,\boldsymbol\cdot\,}\nabla_{\S}\chi_{\eta})
			&=\sum_{\zeta\in Y}\big(\nabla_{\S} \chi_{\xi}(\zeta){\,\boldsymbol\cdot\,}\nabla_{\S}\chi_{\eta}(\zeta)\big)\ \omega_{\zeta}
			\\
			&=\sum_{\zeta\in Y}\Big(\sum_{\beta\in X}\alpha_{\beta,\xi}\nabla_{\S}\phi_m(\zeta,\beta)\Big){\,\boldsymbol\cdot\,}\Big(\sum_{\gamma\in X}\alpha_{\gamma,\eta}\nabla_{\S}\phi_m(\zeta,\gamma)\Big)\ \omega_{\zeta}.
		\end{aligned}
	\end{equation}
	More implementation details for computing $A^Y$, $B^Y$ by using Lagrange functions of conditionally positive definite kernels can be found in \cite[Sec. 9]{narcowich2017novel}.
	Finally, the {fully-discrete nonlinear} equation \eqref{eq:NLWfullDiscretization_local} could be solved by using some iteration algorithms, such as the fixed point iteration method and Newton's method.
	The following theorem establishes the error estimates of the discretized integrals $Q_Y(\chi_{\xi}\chi_{\eta})$ and $Q_Y(\nabla_{\S}\chi_{\xi}\cdot\nabla_{\S}\chi_{\eta})$.

	\begin{theorem}\label{thm:quadratureErr}
		Let the surface $\S$ be a smooth, closed surface satisfies the assumptions in \cite{reeger2016numerical,hangelbroek2018direct}
		and $X,Y\subset \S$ are sets of trial centers and quadrature points.
		Let $m>d_{\S}/2+1$
		be the smoothness order of the Mat\'{e}rn kernels that reproduce the Sobolev space on surface $H^m(\S)$ and Lagrange functions $\chi_{\xi}$ defined on $X$ as in \eqref{eq:LagrangFunc}.
		If  we employ a kernel-based quadrature formula on $Y$ of Reeger \eqref{eq:KernelQuadrature} with approximation power
		\begin{equation}\label{eq:ErrQuadrature}
			\Big|\int_{\S} {f}\d \mu- Q_Y( {f})\Big|\leq C h_Y^{\ell}\| {f}\|_{H^{\ell}(\S)},
		\end{equation}
		for some $\ell\geq m$, and $h_Y$ is the fill distance of $Y$.
		Then there exists a  constant $C$ independent of $X$ and $Y$ so that the following inequalities hold
		$$\Big|\int_{\S}\chi_{\xi}\chi_{\eta}\d \mu-Q_Y(\chi_{\xi}\chi_{\eta})\Big|\leq C(h_Y/h_X)^{m}h_X,$$
		$$\Big|\int_{\S}\nabla_{\S}\chi_{\xi}\cdot\nabla_{\S}\chi_{\eta}\d \mu-Q_Y(\nabla_{\S}\chi_{\xi}\cdot\nabla_{\S}\chi_{\eta})\Big|\leq C(h_Y/h_X)^{m},$$
		for sufficiently small $h_X$ and $h_Y$.
	\end{theorem}
	\begin{proof}
		The case of spherical integrals about the Lagrange functions associated with surface splines was obtained in \cite[7.2 and 7.3]{narcowich2017novel}. The critical success factors are the Bernstein inequality, stability results, and inverse inequalities of Lagrange functions.
		In \cite[Thm. 3, Prop.~9, Thm. 10]{hangelbroek2018direct}, the authors also verified the same theories for Mat\'{e}rn kernels on closed, compact and smooth domains. Here, we give a brief proof, the details can be found in the above references.
		Since  {$m>d_\S/2+1$}, $\chi_{\xi}\in H^{m}(\S)\cap L^{\infty}(\S)$ and $\nabla_{\S}\chi_{\xi}\in H^{m-1}(\S)\cap L^{\infty}(\S)$. Follow the same technique in \cite[Prop. 4.4]{narcowich2017novel}
		and use the stability result and inverse inequality in \cite[Prop.~9, Thm. 10]{hangelbroek2018direct}, we see that $\|\chi_{\xi}\chi_{\eta}\|_{H^{m}}\leq C h_X^{1-m}$.
		From the bound in \eqref{eq:ErrQuadrature} and the condition $m\leq\ell$, we have
		\begin{align*} \Big|\int_{\S}\chi_{\xi}\chi_{\eta}\d \mu-Q_Y(\chi_{\xi}\chi_{\eta})\Big|\leq Ch_Y^{m}\|\chi_{\xi}\chi_{\eta}\|_{H^{m}}
			\leq Ch_Y^{m} h_X^{1-m}=C(h_Y/h_X)^{m} h_X.
		\end{align*}
		The error bound of second inequality can be obtained similarly
		\begin{align*} \Big|\int_{\S}\nabla_{\S}\chi_{\xi}\cdot\nabla_{\S}\chi_{\eta}\d \mu-Q_Y(\nabla_{\S}\chi_{\xi}\cdot\nabla_{\S}\chi_{\eta})\Big|\leq Ch_Y^{m}\|\nabla_{\S}\chi_{\xi}\cdot\nabla_{\S}\chi_{\eta}\|_{H^{m-1}}
			\leq C(h_Y/h_X)^{m}.
		\end{align*}
	\end{proof}

	\begin{remark}
		Error bounds for the case $\ell<m$ are also possible by using the results in \cite[Corollary~7.2]{narcowich2017novel}.
		Furthermore, the impact of the quadrature error on the convergence behavior of the discretized Galerkin equation \eqref{eq:NLWfullDiscretization_local} is analyzed in \cite{narcowich2017novel,kunemund2019high}. We can follow the same arguments presented there to derive error estimates as in \cite[Thm. 7.9]{narcowich2017novel}.
	\end{remark}

	
	\subsection{Sparse schemes with local Lagrange function}
	The evaluation of coefficients $\alpha_{\eta,\xi}$ of global Lagrange functions requires solving a large  {(usually full)} system, which is a time-consuming task for a full matrix.  To reduce the computational expense, we can replace the global Lagrange functions with the local Lagrange functions \cite{fuselier2013localized,narcowich2017novel,kunemund2019high} in assembling the mass matrix $A^Y$ and stiffness matrix $B^Y$ since the coefficients decay exponentially fast in the distance between $\xi$ and $\eta$.
	
	The local Lagrange functions $\{\chi_{\xi}^{\text{loc}}:\xi\in X\}$ were introduced in \cite{fuselier2013localized} to build the finite dimensional approximation space $\trialsp$. The basis function $\chi_{\xi}^{\text{loc}}$ is the Lagrange function obtained only with the points that lie in a ball of radius $Kh_X|\log(h_X)|$ about $\xi$ by an appropriate {parameter $K$}  in the definition \eqref{eq:ConstrLocalLag}. The computation of local Lagrange functions can be implemented in a very efficient parallelizable way. We refer the reader to the above mentioned references for more details about the construction and properties of local Lagrange functions.
	
	With the local Lagrange functions, we construct sparse matrices $A^{Y,\text{loc}}$,  $B^{Y,\text{loc}}$ with entries
	\begin{equation}\label{eq:localMass}
		A^{Y,\text{loc}}_{\xi,\eta}=Q_{Y_{\xi}\cup Y_{\eta}}(\chi_{\xi}^{\text{loc}}\chi_{\eta}^{\text{loc}}),~~B^{Y,\text{loc}}_{\xi,\eta}=Q_{Y_{\xi}\cup Y_{\eta}}(\nabla\chi_{\xi}^{\text{loc}}\cdot\nabla\chi_{\eta}^{\text{loc}}) ,
	\end{equation}
	{for each pair of trial centers $\xi,\eta\in X$.}
	Here $Y_{\xi} {\subset Y}$ is the intersection of $Y$ and  {the support of $\chi_{\xi}$}. Since the discretized matrices $A^{Y,\text{loc}}$ and $B^{Y,\text{loc}}$ are symmetric, we expect that the local Lagrange method still preserve the dissipation properties for general dissipative PDEs.
	To sum up, we present in Algorithm~\ref{Algor:1} the pseudocode for a meshless, structure-preserving scheme based on local Lagrange functions, designed to solve dissipative PDEs on surfaces.

	\begin{algorithm}[H]
		\caption{(Local meshless structure-preserving Galerkin scheme)
		}
		\label{Algor:1}
		\begin{enumerate}
			\item Select a global smoothness order $m$ of the  Mat\'{e}rn kernel in Section \ref{sec:Prelim} such that:
			\begin{itemize}
				\item $m>d_{\S}/2+1$, where $d_{\S}$ is the dimension of the surface $\S$
				\item The resulting surface Sobolev space is $W_2^m(\S)$.
			\end{itemize}
			
			\item Define the sets of:
			\begin{itemize}
				\item Trial centers $X \subset\S$ of $N_X$ quasi-uniform points
				\item Quadrature points $Y \subset\S$ of $N_Y$  quasi-uniform points
			\end{itemize}
			by \emph{distmesh} algorithm in \cite{persson2004simplemesh} or other appropriate means

			\item For some constant $K$  and each trial center $\xi \in X$:
			\begin{itemize}
				\item Identify the stencil/neighborhood $X_\xi \subset X$ of trial centers around $\xi$  with radius $Kh_X|\log(h_X)|$
				
				\item Calculate the coefficients $c_{\xi,  {j}}$ of the local Lagrange functions $\chi_{\xi}$ in \eqref{eq:ConstrLocalLag} using the nodes in $X_\xi$
				\item Identify local quadrature points $Y_{\xi}= Y \,\cap\, \text{supp}(\chi_\xi)$ associated with $\xi$
			\end{itemize}
			
			\item For each pair of quadrature points $\xi,\eta \in X$:
			\begin{itemize}
				\item Utilize the kernel quadrature formula introduced in Section~\ref{sec:KenelQuad} to obtain the quadrature weights  $\{ \omega_{ {\xi\eta,j}} \}$ for each point in $Y_\xi\cap Y_\eta$
				\item Compute  the $\xi\eta$- and $\eta\xi$-entries in the local mass matrix $A^{Y,\text{loc}}$ based on equation \eqref{eq:localMass}
				\item Construct the $\xi\eta$- and $\eta\xi$-entries in the local stiffness matrix $B^{Y,\text{loc}}$ based on equations \eqref{eq:localMass} and  \eqref{eq:ComputStiffMatrix}
			\end{itemize}
			
			\item Solve the fully-discrete equation \eqref{eq:ACfullDiscretization} by replacing $A$, $B$  with the approximated matrices $A^{Y,\text{loc}}$, $B^{Y,\text{loc}}$ using an appropriate iterative method for the nodal solution values (i.e., unknown coefficients) in \eqref{alphadef} for all time levels $n$
		\end{enumerate}
	\end{algorithm}

	
	\section{Numerical experiments}\label{sec:NumerExp}

	In this section, we explore the convergence results and energy dissipation properties of our proposed meshless Galerkin method, which solves both the Allen-Cahn and the Cahn-Hilliard equation on surfaces. We compute the integral over different surfaces using an eighth-order kernel-based quadrature formula. To measure the error at the time $T$,
	we use some sets of quasi-uniform and sufficiently dense sets of evaluation point $Z$ to define the relative 
	$L_2(\S)$-norm error of numerical solution $u_h$ to the exact solution $u$ by
	\begin{equation*}
		\frac{\big\|u(\cdot,T)-u_h(\cdot,T)\big\|_{L_2(\S)}}{\|u(\cdot,T)\|_{L_2(\S)}}
		\approx\left( \frac{ \sum_{\eta\in Z }\omega_{\eta}|u(\eta,T)-u_h(\eta,T)|^2 }{  \sum_{\eta\in Z }\omega_{\eta}|u(\eta,T)|^2}\right)^{1/2}.
	\end{equation*}
	
	To apply the results of the theorems presented in this paper, we must first define a Riemannian metric on the surface $\S$. This is theoretically feasible since $\S$ is assumed to be a smooth manifold. Next, we need to compute the explicit form of the Mat\'{e}rn kernel, which involves solving for the fundamental solutions of high-order Laplace-Beltrami operators on $\S$. However, finding this fundamental solution poses a significant challenge because we often lack a closed-form expression for an atlas of $\S$, making it difficult to solve the equation in coordinate representations. Consequently, constructing an intrinsic kernel on $\S$ can be prohibitively difficult \cite{Powell-2022-Koopman}. Despite these challenges, if the kernel is available, the approximations and theorems discussed in this paper are directly applicable. 
	
	In specific examples and numerical studies, it may be practical to consider extrinsic approximations instead. This involves defining a kernel over the embedded manifold by restricting a kernel defined on the ambient Euclidean space to the surface. This approach has the advantage of not requiring precise knowledge of the surface $\S$; we only need to ensure that the inputs to our kernel originate from this restriction rather than from the entire ambient space. The primary drawback of the restriction method lies in the potential loss of smoothness when approximating functions.  Let $\psi_m(r)$ be defined as:
	\begin{equation}\label{eq:MaternKernel}
		\psi_m(r)=C_m(\epsilon r)^{m-d/2}K_{m-d/2}(\epsilon r),
	\end{equation}
	where $m>d/2$ is the smoothness order, $K_{\nu}$ is the modified Bessel function of the second kind, and $C_m=2^{1-(m-d/2)}/\Gamma(m-d/2)$ normalizes the kernel.
	By restricting the kernel to the surface \cite{narcowich2007approximation,Fuselier+Wright-ScatDataInteEmbe:12},
	we get a reproducing kernel for the Sobolev space $W_2^{s}(\S)$ with smoothness order $s=m-1/2>d_\S/2$.

	\subsection{Full Lagrange basis function}
	In this section, we present numerical evidence demonstrating that the full Lagrange function and the associated coefficients, constructed from the Mat\'{e}rn kernel $\psi_m$ restricted to the sphere $\mathbb{S}^2$, decay exponentially away from their center. Using $\psi_m$, the full Lagrange function centered at $\xi$ is given by $$\chi_{\xi}(x) = \sum_{\eta\in X}\alpha_{\xi,\eta}\psi_m(\|x-\eta\|),~~x,\eta\in\mathbb{S}^2.$$ For comparison, we also consider the conditional positive definite surface splines $$\varphi_k(x,y)=(1-x\cdot y)^{k-1}\log(1-x\cdot y),$$ which have been extensively explored in the literature \cite{fuselier2013localized,narcowich2017novel,kunemund2019high,lehoucq2016meshless}. The corresponding Lagrange function is expressed as: $$\chi_{\xi}(x) = \sum_{\eta\in X}\alpha_{\xi,\eta}\varphi_k(x,\eta)+p_{\xi}(x),~~x,y\in\mathbb{S}^2,$$ where $p_{\xi}$ is a certain degree of spherical harmonic.
	
	In our test, we use minimal energy points as centers for the sphere and evaluate the Lagrange function $\chi_{\xi}$ at the north pole $\xi=(0,0,1)$. The function $\chi_{\xi}(x)$ is computed on a set of gridded points in the parameter space $[\theta_1,\theta_2]$ with $152$ equispaced longitudes $\theta_1\in[0,2\pi]$ and $179$ equispaced colatitudes $\theta_2\in[0,\pi]$ (\cite{fuselier2013localized}). The left figure in Figure \ref{fig.LagrangeValue} shows the values of full Lagrange basis function for the Mat\'{e}rn kernel $\psi_3$ and the surface spline $\varphi_2$ for sets of centers with sizes $N_X=900$ and $N_X=2500$. The right figure displays the coefficients $\alpha_{\xi,\eta}$ of the Mat\'{e}rn kernel.
	This figure demonstrates numerically that both the full Lagrange basis function $\chi_{\xi}$ and the Lagrange coefficients $\alpha_{\xi,\eta}$ of the restricted Mat\'{e}rn kernel decay exponentially away from their centers, similar to the behavior observed for the surface splines in Figure \ref{fig.LagrangeValue}(a) that are detailed in \cite[Sec.~4]{fuselier2013localized}. Hence, we will employ the restricted kernel in the following sections.

	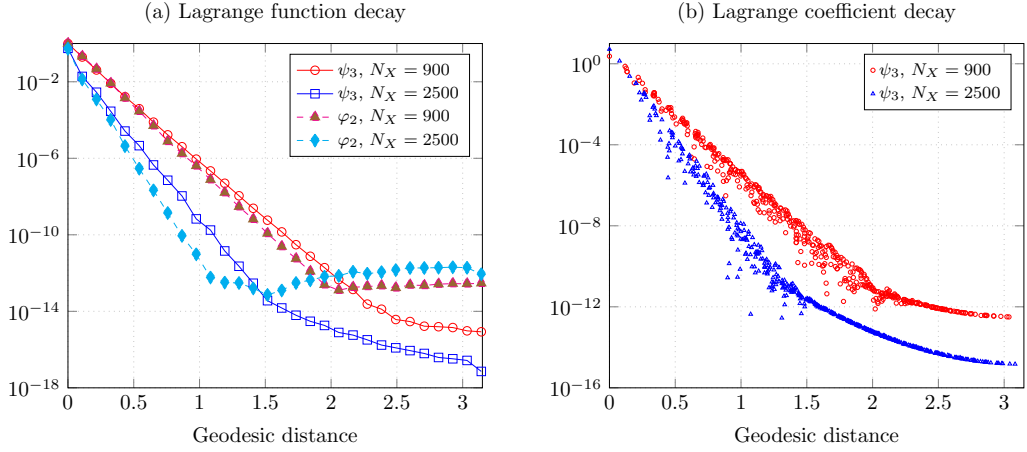
\begin{figure}
		\centering
		\vspace{1cm}
		\begin{tikzpicture}[scale=0.8]
			\begin{semilogyaxis}[
				grid=major,
				grid style = dotted,
				mark size = 1.5pt,
				xmin = 0, xmax = 3.15,
				ymin = 10^(-18),ymax = 10^(-0),
				xlabel={Geodesic distance},
				title={(a) Lagrange function decay},
				legend cell align = {left},
				legend pos = north east,
				legend style = {font=\footnotesize},
				legend entries={{$\psi_3$, $N_X=900$},{$\psi_3$, $N_X=2500$},{ $\varphi_2$, $N_X=900$},{$\varphi_2$, $N_X=2500$}}]
				\addplot  +[mark = o,  color = red, mark size = 2pt] table [x=dist,y=N1] {Lag_func_value_matern.dat};
				\addplot  +[mark = square,  color = blue, mark size = 2pt] table [x=dist,y=N2] {Lag_func_value_matern.dat};
				\addplot  +[dashed, mark = triangle*,   color = magenta, mark size = 3pt] table [x=dist,y=N1] {Lag_func_value.dat};
				\addplot  +[dashed,mark = diamond*, color = cyan, mark size = 3pt] table [x=dist,y=N2] {Lag_func_value.dat};
			\end{semilogyaxis}
		\end{tikzpicture}
		\hspace{0.5cm}
		\begin{tikzpicture}[scale=0.8]
			\begin{semilogyaxis}[
				grid=both,
				grid style = dotted,
				mark size = 1.5pt,
				xlabel={Geodesic distance},
				title={(b) Lagrange coefficient decay},
				ymin = 10^(-16),ymax = 10^1,
				xmin = 0, xmax = 3.15,
				legend cell align = {left},
				legend pos = north east,
				legend style = {font=\footnotesize},
				legend entries={{$\psi_3$, $N_X=900$},{$\psi_3$, $N_X=2500$}}]
				\addplot  +[ only marks, mark= o,mark size = 1pt, color=red] table [x=dist,y=coeff] {Lag_coeff_value_matern_N900.dat};
				\addplot  +[only marks, mark = triangle, mark size = 1pt, color =blue] table [x=dist,y=coeff] {Lag_coeff_value_matern_N2500.dat};
			\end{semilogyaxis}
		\end{tikzpicture}
		\vspace{-1pt}
		\caption{(a) Maximum latitudinal values of the Lagrange function for the Mat\'{e}rn kernel $\phi_3$ and the surface spline $\psi_2$; (b) The Lagrange coefficients $\alpha_{\xi,\eta}$ constructed from the restricted Mat\'{e}rn kernel $\phi_3$ with the shape parameter $\epsilon=14$.}
		\label{fig.LagrangeValue}
	\end{figure}
	
	\subsection{Accuracy and convergence test}
	\label{subsec:AC_accuracy}
	In this example, our goal is to evaluate the spatial and temporal accuracy, as well as the convergence rate, of our proposed meshless Galerkin method for solving the Allen-Cahn equation on both the unit sphere and a torus. The equation features a source term:
	\begin{equation}\label{eq:ParaEq}
		u_t-\varepsilon^2\Delta_{\S}u+u^3-u=f(x,t),~~\text{on} ~(0,T]\times\S,
	\end{equation}
	where $f(x,t)$ and the initial conditions are fixed by given exact solutions:
	\begin{equation}\label{eq:ACEqSoluSph}
		u_1(x,t)=\tanh(x_1+x_2+x_3-t),~~x=(x_1,x_2,x_3)\in \Sph^2,
	\end{equation}
	on the unit sphere, or
	\begin{equation}\label{eq:ACEqSoluTorus}
		u_2(x,t)=\frac{1}{8}e^{-5t}x_1(x_1^4-10x_1^2x_2^2+5x_2^4)(x_1^2+x_2^2-60x_3^2),
		~~x=(x_1,x_2,x_3)    \in \mathbb{T}^2,
	\end{equation}
	on a torus $\mathbb{T}^2:=\{x\in \R^3\Big|\big(1-\sqrt{x_1^2+x_2^2}\big)^2+x_3^2-\frac{1}{9}=0\}$.
	These specific functions are frequently employed in literature, for example, as seen in \cite{mohammadi2019numerical,fuselier2013high,fuselier2015order}.
	
	\begin{figure}[tp]
		\centering
		\vspace{0.5cm}
		\begin{tabular}{ccccc}
			\begin{overpic}[width=0.45\textwidth,trim= 0 0 0 0, clip=true,tics=10]{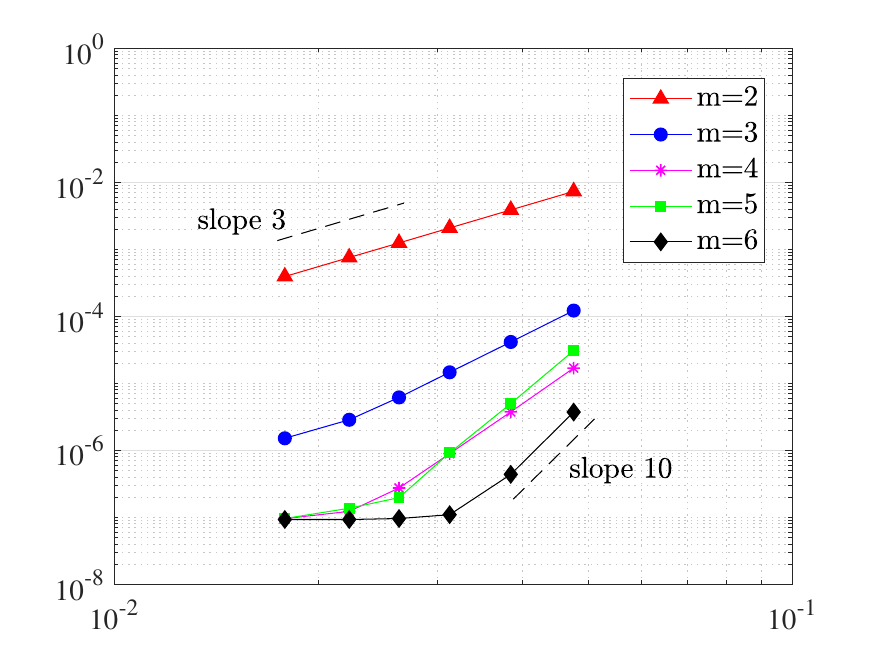}
			\put (36,-1) {$1/\sqrt{N_X}\sim h_X$}
			\put (0,32) {\scriptsize \rotatebox{90}{$L_2$ error}}
			\put (46,75) {\scriptsize Sphere}
		\end{overpic}
		&
		\begin{overpic}[width=0.45\textwidth,trim= 0 0 0 0, clip=true,tics=10]{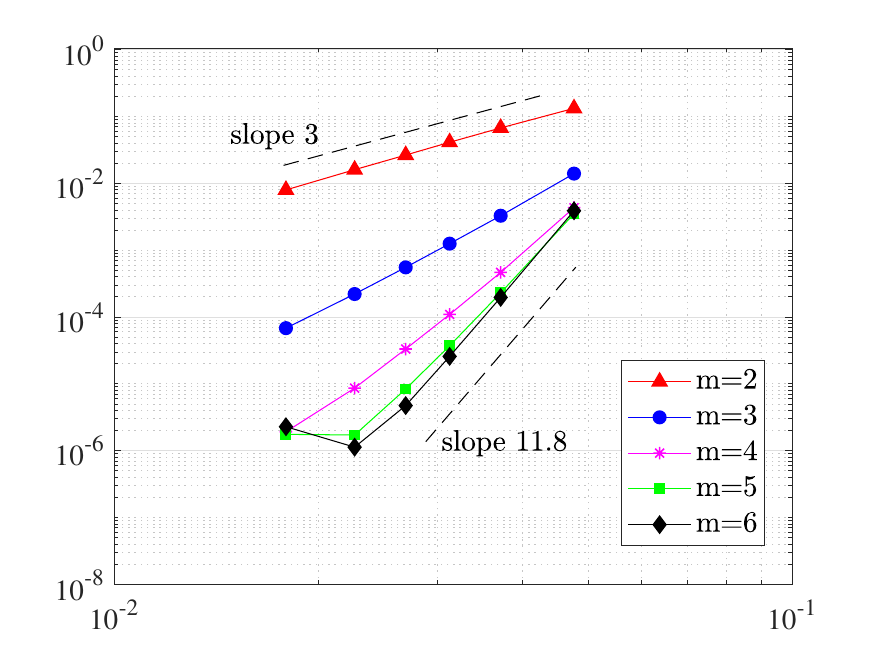}
		\put (36,-1) {$1/\sqrt{N_X}\sim h_X$}
		\put (0,32) {\scriptsize \rotatebox{90}{$L_2$ error}}
		\put (48,75) {\scriptsize Torus}
	\end{overpic}
\end{tabular}
\caption{
	The relative $L_2(\S)$-error profiles with respect to $h_X$ for the proposed sparse kernel-based Galerkin method in Algorithm~\ref{Algor:1}. This method is tested on the Allen-Cahn equation defined on a sphere and a torus, using a local Lagrange basis constructed by the Matérn kernel of varying smoothness orders ($m=2,3,4,5,6$). The quadrature points ($Y$) are held constant in these tests, with $N_Y=40001$ for the sphere and $N_Y=39122$ for the torus.
}
\label{Fig.AC_L2Err_NX}
\end{figure}

\begin{figure}[t]
\vspace{0.5cm}
\centering
\begin{tabular}{ccccc}
	\begin{overpic}[width=0.45\textwidth,trim= 0 0 0 0, clip=true,tics=10]{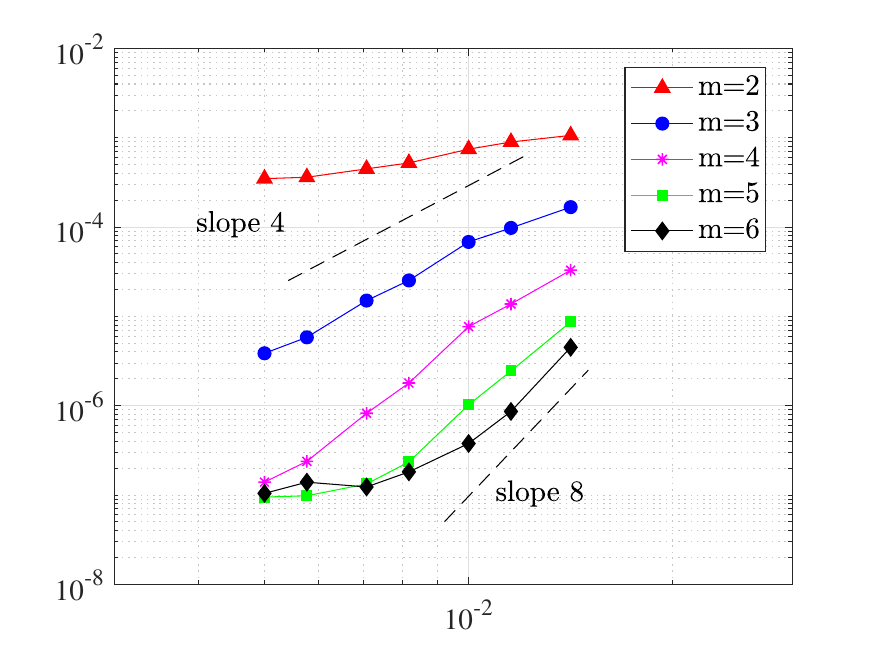}
	\put (36,-3) {$1/\sqrt{N_Y}\sim h_Y$}
	\put (0,32) {\scriptsize \rotatebox{90}{$L_2$ error}}
	\put (46,75) {\scriptsize Sphere}
\end{overpic}
&
\begin{overpic}[width=0.45\textwidth,trim= 0 0 0 0, clip=true,tics=10]{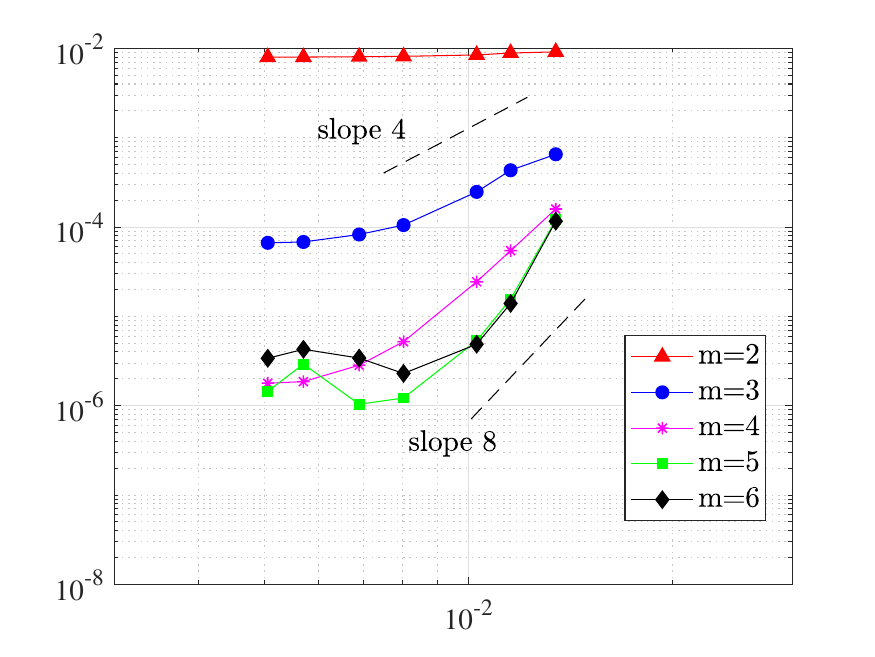}
\put (36,-3) {$1/\sqrt{N_Y}\sim h_Y$}
\put (0,32) {\scriptsize \rotatebox{90}{$L_2$ error}}
\put (48,75) {\scriptsize Torus}
\end{overpic}
\end{tabular}
\caption{
As a counterpart to Figure~\ref{Fig.AC_L2Err_NX}, this figure presents the relative $L_2(\S)$-error profiles with respect to $h_Y$. Here, $N_X$ remains constant with $N_X=3721$ for the sphere and $N_X=3112$ for the torus.
}
\label{Fig.AC_L2Err_Ny}
\end{figure}
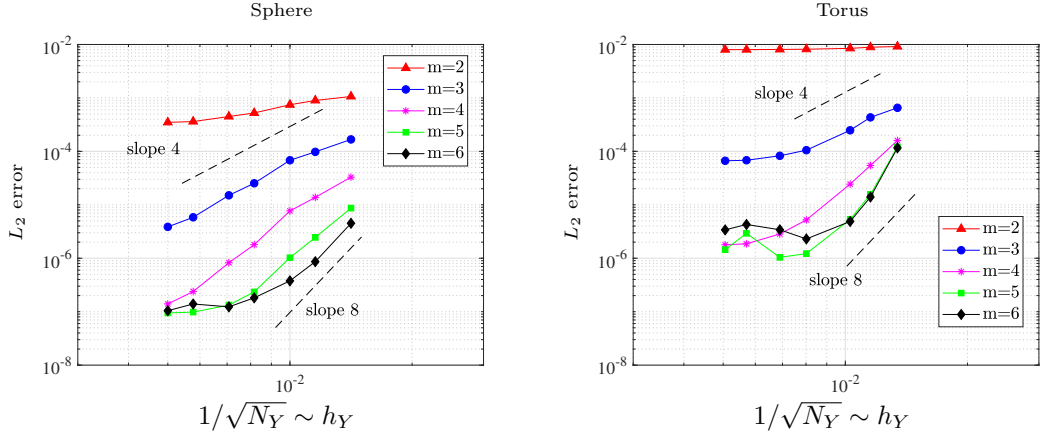

We employ Algorithm~\ref{Algor:1} to solve \eqref{eq:ParaEq} with $\varepsilon=1$ and different (global) smoothness orders $m=2,3,4,5,6$. We set a small time step $\triangle t=10^{-8}$ and a final time $T=0.0001$ (as per \cite{mohammadi2019numerical}) to ensure the global error is primarily driven by spatial discretization. As outlined in Theorem \ref{thm:SpatialErr} and Theorem \ref{thm:quadratureErr}, the spatial error comprises two components: approximation error and quadrature error. To prioritize the influence of approximation error on spatial error, we choose a large fixed number of quadrature points $N_Y$ and a relatively small value for $N_X$.

Figure~\ref{Fig.AC_L2Err_NX} shows the $L_2(\S)$ errors on the sphere and torus relative to $h_X\approx 1/\sqrt{N_X}$ using different smoothness orders. For the unit sphere, we employ minimum energy node sets of varying sizes ($N=441$, $676$, $1024$, $1444$, $2025$, and $3136$). For the torus, we generate node sets using a mesh generator \cite{persson2004simplemesh}, with sizes $N=440$, $724$, $1024$, $1380$, $1952$, and $3112$.
The convergence rates of the numerical solutions for the Allen-Cahn equation on both the sphere and the torus significantly exceed the theoretical rates of $\mathcal{O}(h_X^{m-1/2})$. This phenomenon, known as super-convergence, has been observed in various studies \cite{fuselier2013high,fuselier2015order,sun2022kernel,chen2023kernel}.

To study the effect of the local Lagrange function's support size, we observe that the number of points within the radius $r_X=Kh_X|\log(h_X)|$ will be $n:=|X_{\xi}|\sim K^2(\log(N_X))^2$. Additionally, the theoretical considerations and numerical results in \cite{fuselier2013localized} suggest that selecting the number of nearest neighbors as $n=K^2\lceil(\log_{10}(N_X))^2\rceil$ with $K^2=7$ yields satisfactory results for constructing local Lagrange functions. Our numerical experiments in Table \ref{tab:EffectOfParameterK} suggest that using $K^2\geq 1$ on the unit sphere and $K^2\geq 3$ on the torus can also provide comparable accuracy. It's important to consider, however, that opting for larger $M=K^2$ values could lead to higher computational cost, despite potentially enhancing accuracy.

\begin{table}[t]
\centering
\caption{Computation times and numerical errors obtained by applying Algorithm~\ref{Algor:1} to the Allen-Cahn equation \eqref{eq:ParaEq} with varying $n$ values. For the unit sphere, $N_X=3136$ and $N_Y=10001$ are used, while for the torus, $N_X=3124$ and $N_Y=9470$ are used.}
\label{tab:EffectOfParameterK}
\begin{tabular}{*{6}{c}}
\toprule
\multirow{2}{*}{$K^2$} & \multirow{2}{*}{$n$} & \multicolumn{2}{c}{Sphere} &\multicolumn{2}{c}{Torus}\\
\cline{3-6}
\multicolumn{2}{c}{}& $L_2$ error & CPU(s) & $L_2$ error & CPU(s)  \\
\midrule
1& $13$&  4.14227e-7 &    191.56       &  -  & -\\
2& $26$&  4.14219e-7 &    205.98      &  -  & -\\
3 & $39$& 4.15629e-7 &    229.30 & 3.30729e-6   &  222.64    \\
5 & $65$&  4.14221e-7 &    243.13 &  3.30195e-6 &  231.81   \\
7& $91$& 4.14228e-7 &    261.92   &  3.30181e-6  &  241.50 \\
11& $143$ &  4.14228e-7 & 292.56& 3.30176e-6  &  272.48   \\
\bottomrule
\end{tabular}
\end{table}

To numerically estimate the quadrature error bound, we use a fixed number of centers $N_X$. Figure \ref{Fig.AC_L2Err_Ny} shows the $L_2(\S)$ errors in relation to the fill distance $h_Y$. Fibonacci node sets of various sizes $N_Y=5001$, $7501$, $10001$, $15001$, $20001$, $30001$, $40001$ are used for quadrature on the sphere. Likewise, node sets generated by \cite{persson2004simplemesh} of differing sizes $N_Y=5534$, $7520$, $9470$, $15560$, $21030$, $30704$, $39122$ are employed on the torus.
As expected, the observed convergence rates outperform the estimated rates in Theorem \ref{thm:quadratureErr}. Furthermore, there is hardly any convergence for $m=2$, which confirms the necessity of the requirement in Theorem \ref{thm:quadratureErr}.

For the temporal accuracy test, we solve the equation \eqref{eq:ParaEq} up to time $T=1$ with various time steps by utilizing Mat\'{e}rn kernel with $m=5$ (shape parameter $\epsilon=8$). The error profiles are presented in Table \ref{tab:ACEq_LLM_timetest_sph} for both the unit sphere ($N_X=961$, $N_Y=30001$) and the torus ($N_X=1024$, $N_Y=30704$). The numerical results in Table \ref{tab:ACEq_LLM_timetest_sph} demonstrate that the AVF method did yield a second order of convergence in time for both surfaces as expected.

\begin{table}[t]
\centering
\caption{Time accuracy test results for the Allen-Cahn equation \eqref{eq:ParaEq} up to $T=1$, computed using Algorithm~\ref{Algor:1}. For the unit sphere, $N_Y=30001$ is used, while for the torus, $N_Y=30704$ is used.}
\label{tab:ACEq_LLM_timetest_sph}
\begin{tabular}{*{5}{c}}
\toprule
$\triangle t$& $N_X=961$, sphere &rate   & $N_X = 1024$, torus & rate   \\
\midrule
$0.04$&    1.1439e-4 &        &  5.3426e-3  & \\

$0.02$&    2.8583e-5 &  2.00  &  1.3403e-3  & 1.99  \\

$0.01$&    7.1439e-6 &  2.00  &  3.3603e-4  & 2.00 \\

$0.005$&   1.7874e-6 &  2.00  &  8.6626e-5  & 1.96 \\

$0.0025$&  4.5772e-7 &  1.97  &  2.2519e-5  & 1.94 \\
\bottomrule
\end{tabular}
\end{table}

\subsection{Comparison with meshfree Galerkin methods on the Allen-Cahn equation}\label{subsec:AC}
We perform a series of numerical experiments on both the unit sphere and the torus to evaluate the stability and energy dissipation properties of our proposed method when solving the Allen-Cahn equation
$$u_t=\Delta_{\S}u+\frac{1}{\varepsilon^2}u(1-u^2).$$
These experiments focus on mean curvature flow and phase separation.


First, we simulate the circular interface motion with an initial condition $u_0(x)$ that's set to $+1$ within a spherical cap of radius $r_0$ and $-1$ elsewhere. We use the parameters from \cite{choi2015motion}, selecting $r_0=\frac{1}{\sqrt{2}}$ and $\epsilon=0.05$, and solve the Allen-Cahn equation with a time step of $\triangle t=5\times 10^{-4}$ until we reach a final time $T=0.35$. In this context, the radius at time $t$ is denoted as $r(t)$ and calculated using the formula from \cite{choi2015motion}: $r(t)=[1-(1-r_0^2)e^{2t}]^{1/2}$.

The progression of the diffuse interface, approximated by Algorithm~\ref{Algor:1} ($N_X=3721, N_Y=40001$), is shown in Figure~\ref{fig.AC_init1}. Over time, we observed that the radius of the spherical cap decreased.
In Figure~\ref{fig:AC_Init1_Ener}~(Left), we compare the numerical radius with the analytical radius. Our method shows greater accuracy than both the narrow band method \cite[Fig. 4]{choi2015motion} and the surface finite element method \cite[Fig. 10]{xiao2017stabilized}.
We also calculate the discrete total energy $\H_{d,G}^n$, and present the scaled results $\H_{d,G}^n/\H_{d,G}^0$ in Figure~\ref{fig:AC_Init1_Ener}~(Right). This confirms that the total discrete energy is non-increasing, supporting the conclusion drawn in Theorem \ref{thm:FullDiscrete}.

\begin{figure}[t]
\vspace{1cm}
\begin{tabular}{ccccc}
\begin{overpic}[width=0.22\textwidth,trim= 50 35 60 35, clip=true,tics=10]{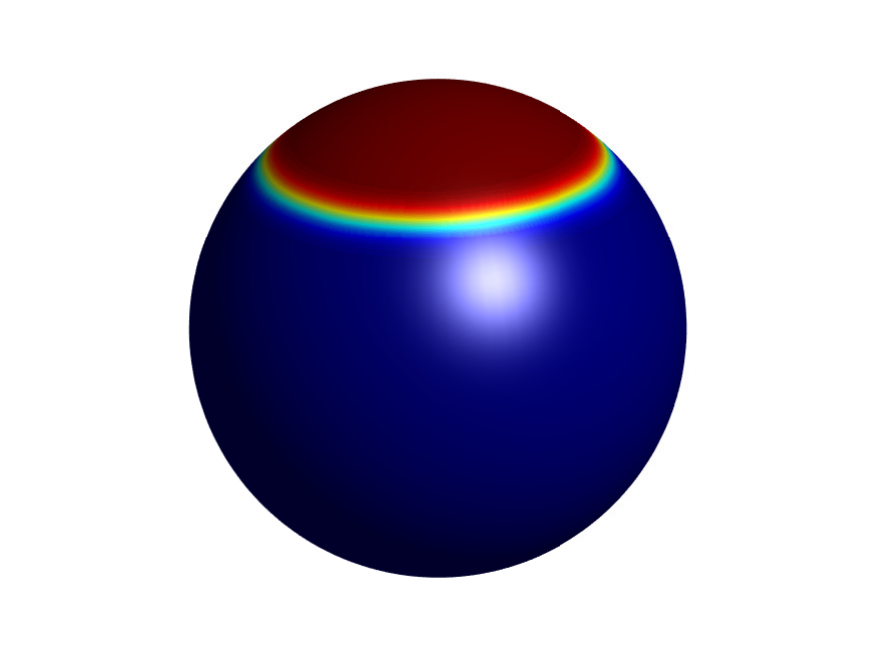}
\put (38,85) {\footnotesize $t=0.01$}
\end{overpic}
&
\begin{overpic}[width=0.22\textwidth,trim= 50 35 60 35, clip=true,tics=10]{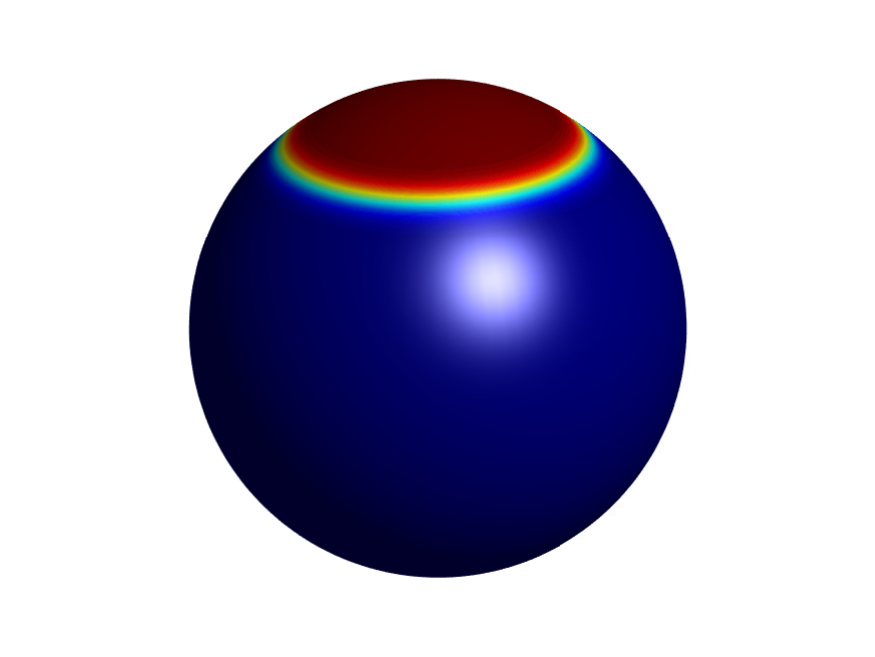}
\put (44,85) {\footnotesize $0.1$}
\end{overpic}
&
\begin{overpic}[width=0.22\textwidth,trim= 50 35 60 35, clip=true,tics=10]{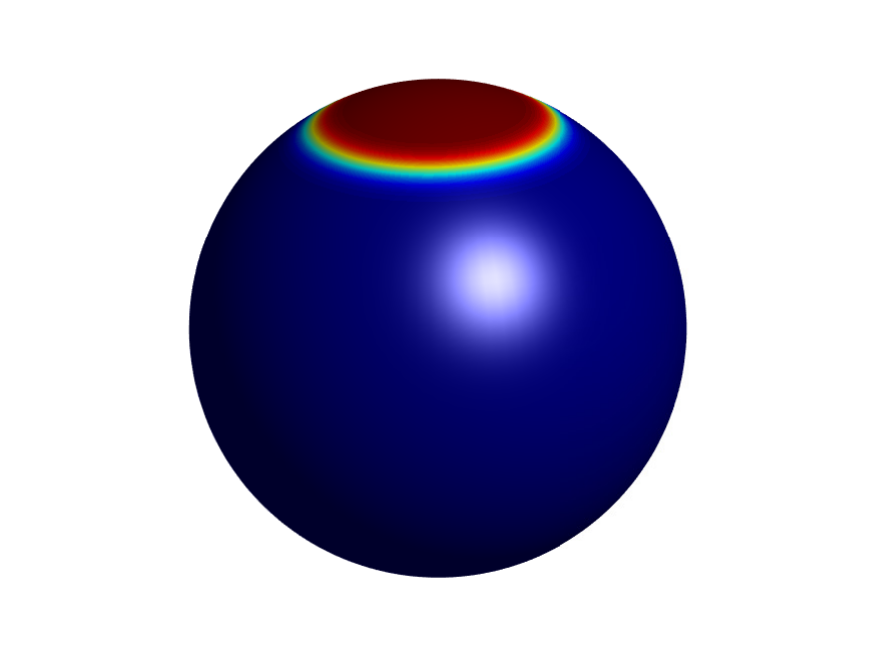} 
\put (44,85) {\footnotesize $0.2$}
\end{overpic}
&
\begin{overpic}[width=0.22\textwidth,trim= 50 35 60 35, clip=true,tics=10]{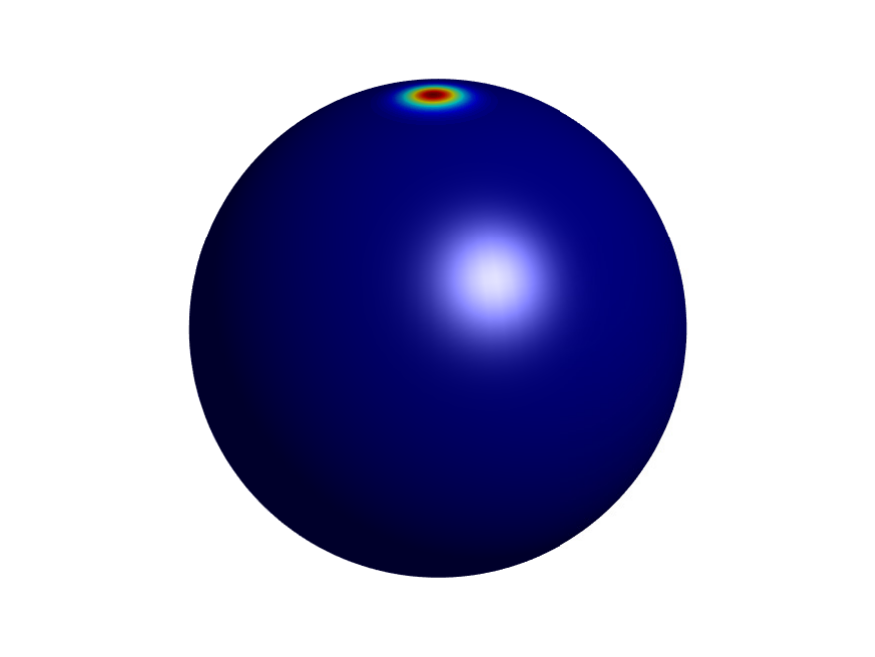} 
\put (42,85) {\footnotesize $0.35$}
\end{overpic}
\end{tabular}
\caption{
Numerical simulations of mean curvature flow modeled by the Allen-Cahn equation at different time levels, approximated by Algorithm~\ref{Algor:1}. The parameters used are $\epsilon=0.05$ and $\triangle t=5\times 10^{-4}$. As time evolves, the interfaces shrink by mean curvature, indicating a decreasing radius of the spherical cap.
}
\label{fig.AC_init1}
\end{figure}

\begin{figure}[htbp]
\vspace{0.3cm}
\centering
\begin{tikzpicture}[scale=0.8]
\begin{axis}[height=2.75in,
grid = major,
xmin=0, xmax=0.4,
ymin=0, ymax=0.8,
xtick={0,0.1,0.2,0.3,0.4},
ytick={0,0.2,0.4,0.6,0.8},
major grid style={line width=.1pt,draw=gray!50},
no markers,
line width = 1pt,
legend pos = north east,
xlabel={$t$}, ylabel={Radius},
legend entries={Analytic,$N_Y=40001$}]
\addplot  + [color = blue] table [x=t1,y=r1] {AC_radius.dat};
\addplot +[dashed, color = red] table [x=t2,y=r2] {AC_radius.dat};
\end{axis}
\end{tikzpicture}
\hspace{0.5cm}
\begin{tikzpicture}[scale=0.8]
\begin{axis}[height=2.75in,
grid = major,
xmin=0, xmax=0.4,
ymin=0, ymax=1,
xtick={0,0.1,0.2,0.3,0.4},
ytick={0,0.2,0.4,0.6,0.8,1},
major grid style={line width=.1pt,draw=gray!50},
line width=1.0pt,
legend pos = north east,
xlabel={$t$}, ylabel={Energy},
]
\addplot  + [solid, mark size = 1pt, color = blue] table [x=t, y=ener] {AC_Init1.dat};
\end{axis}
\end{tikzpicture}
\vspace{-1pt}
\caption{
Comparison of numerical and analytical radii demonstrating the accuracy of the proposed method (left), and scaled representation of non-increasing total discrete energy over time (right), as per Algorithm~\ref{Algor:1} for the Allen-Cahn equation solution.
}
\label{fig:AC_Init1_Ener}
\end{figure}
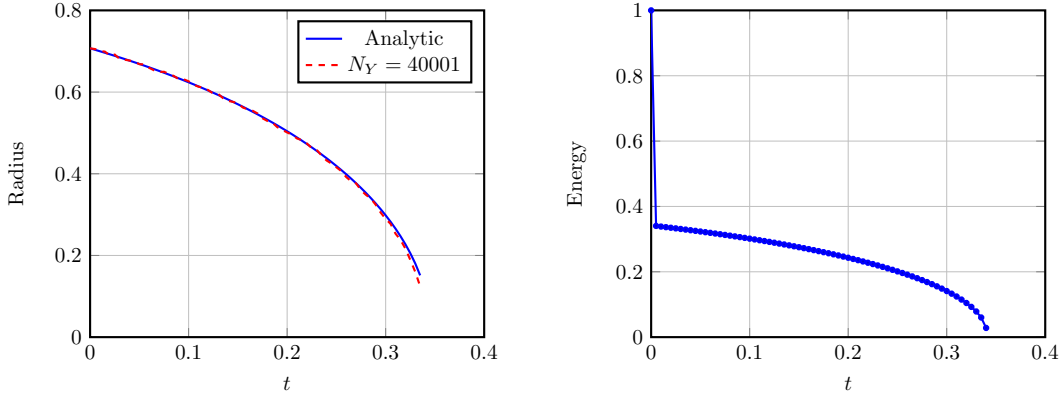

\begin{figure}[htbp]
\centering
\vspace{1cm}
\begin{tabular}{ccccc}
\begin{overpic}[width=0.45\textwidth,trim= 0 0 0 0, clip=true,tics=10]{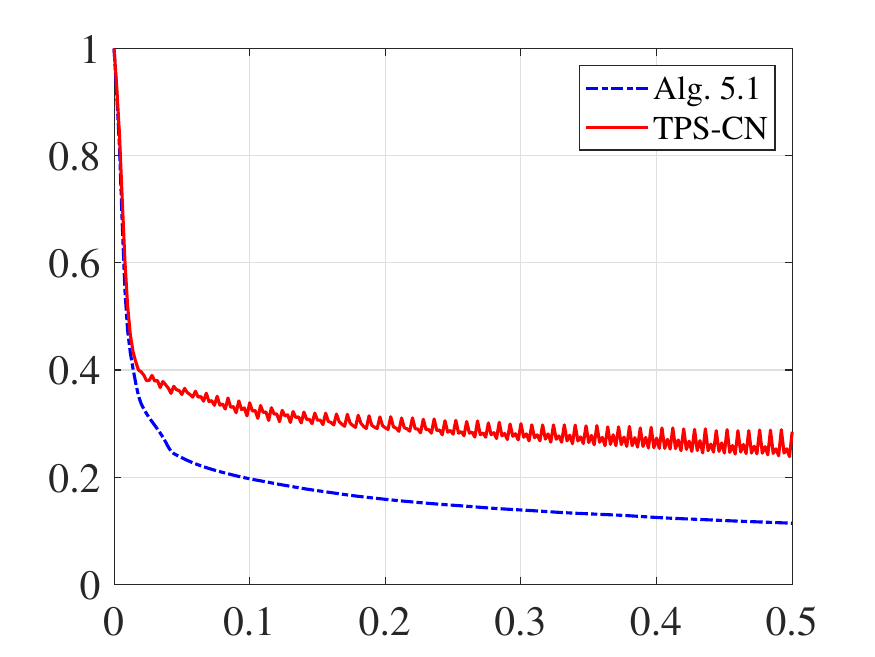}
\put (52,-2) {$t$}
\put (0,32) {\scriptsize \rotatebox{90}{Energy}}
\put (38,75) {\scriptsize $\triangle t=2\times10^{-3}$}
\end{overpic}
&
\begin{overpic}[width=0.45\textwidth,trim= 0 0 0 0, clip=true,tics=10]{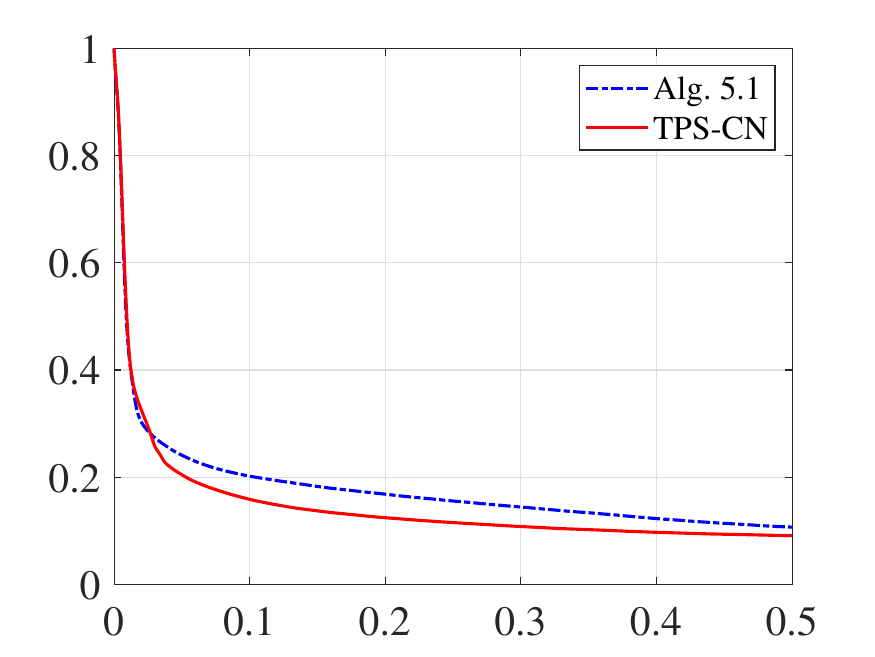}
\put (52,-2) {$t$}
\put (0,32) {\scriptsize \rotatebox{90}{Energy}}
\put (38,75) {\scriptsize $\triangle t=5\times10^{-4}$}
\end{overpic}
\end{tabular}
\caption{
Comparison of scaled discrete total energy of the Allen-Cahn equation with the random initial condition \eqref{eq:AC_RandInit} on the unit sphere. The results of Algorithm~\ref{Algor:1} and the TPS-CN method from \cite{kunemund2019high} are shown for two different time step sizes $\triangle t=2\times 10^{-3}$ and $5\times 10^{-4}$. The non-increasing energy of Algorithm~\ref{Algor:1} and the oscillatory energy of the TPS-CN method are highlighted. }\label{fig.AC_randomInit_enerComp}
\end{figure}

\begin{figure}[htbp]
\vspace{-1cm}
\centering
\begin{tabular}{ccccc}
\begin{overpic}[width=0.45\textwidth,trim= 75 55 55 10, clip=true,tics=10]{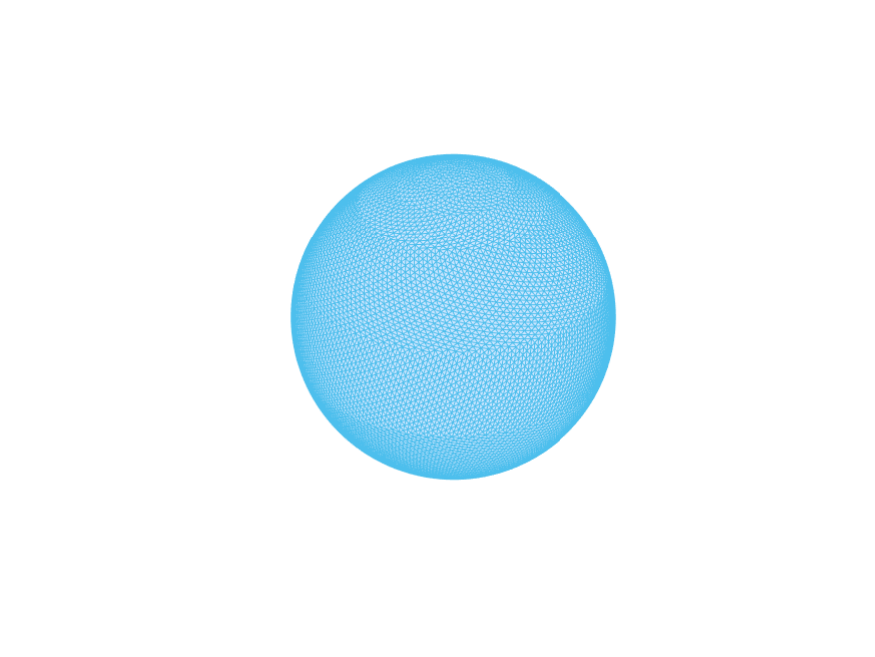}
\end{overpic}
&
\begin{overpic}[width=0.45\textwidth,trim= 85 65 65 55, clip=true,tics=10]{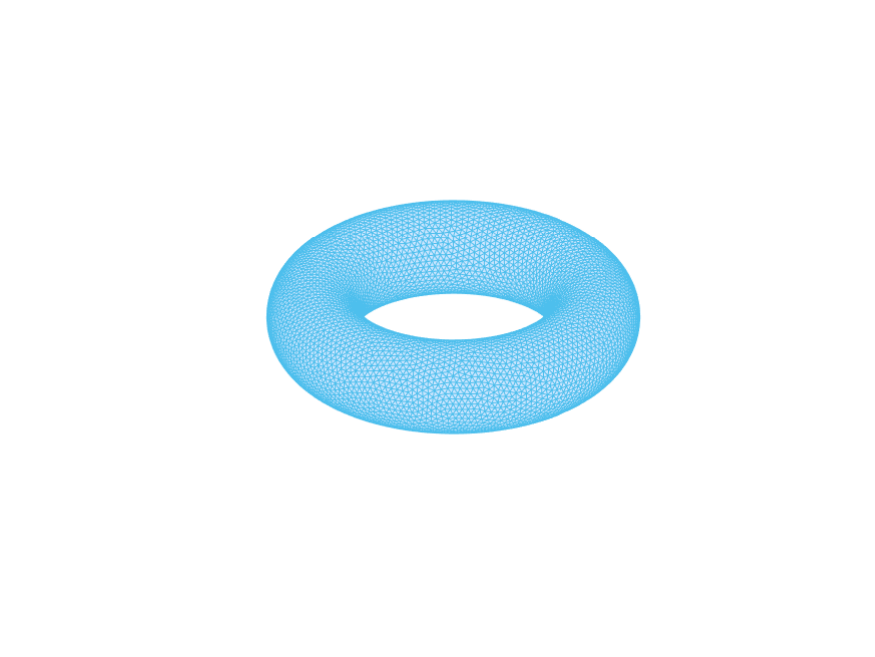}
\end{overpic}
&\vspace{-0.6cm}\\
\begin{overpic}[width=0.4\textwidth,trim = 10 0 20 5, clip=true,tics=10]{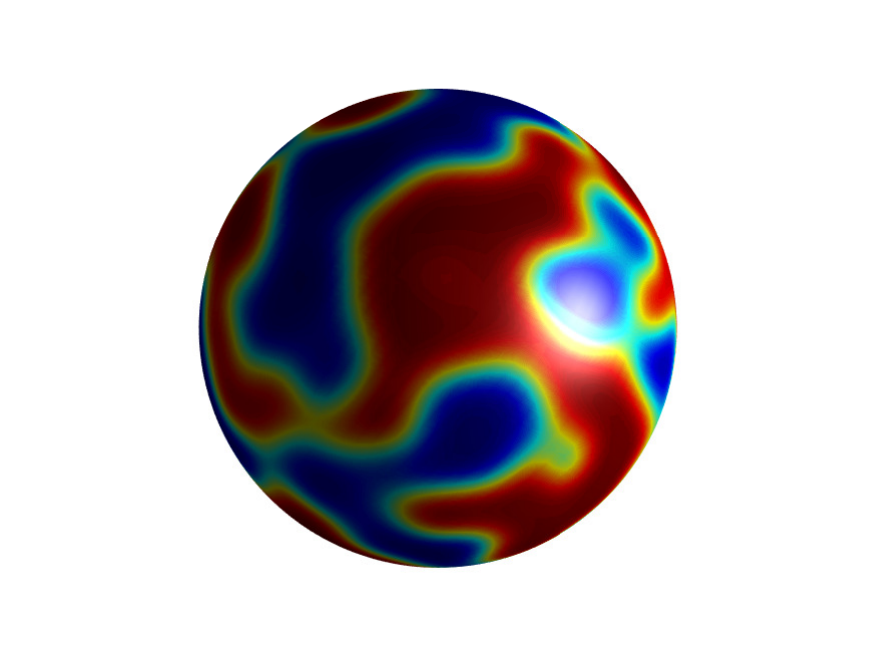} 
\end{overpic}
&
\begin{overpic}[width=0.4\textwidth,trim = 70 50 50 10, clip=true,tics=10]{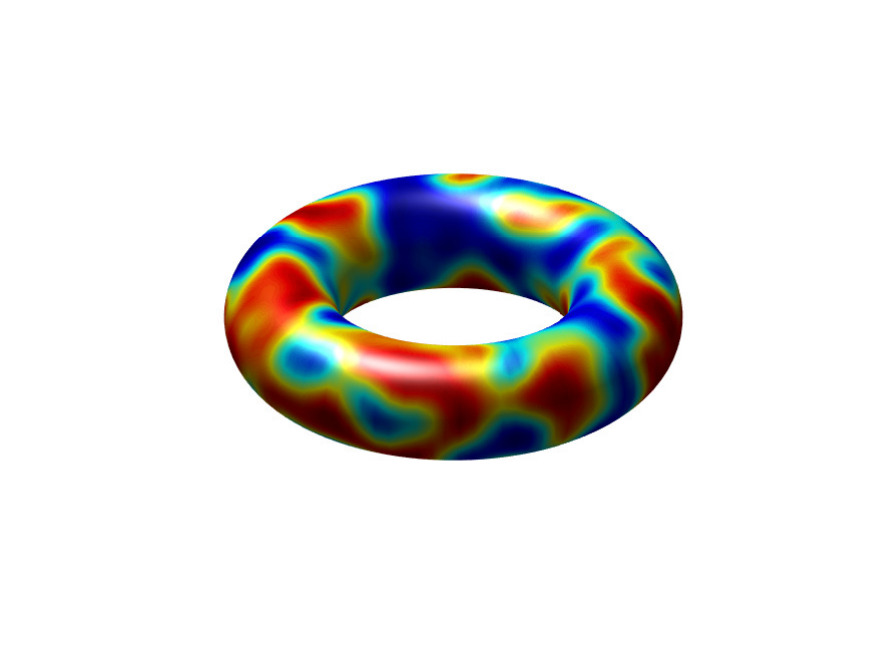} 
\end{overpic}
&\vspace{-0.6cm}\\
\begin{overpic}[width=0.4\textwidth,trim = 10 0 20 5, clip=true,tics=10]{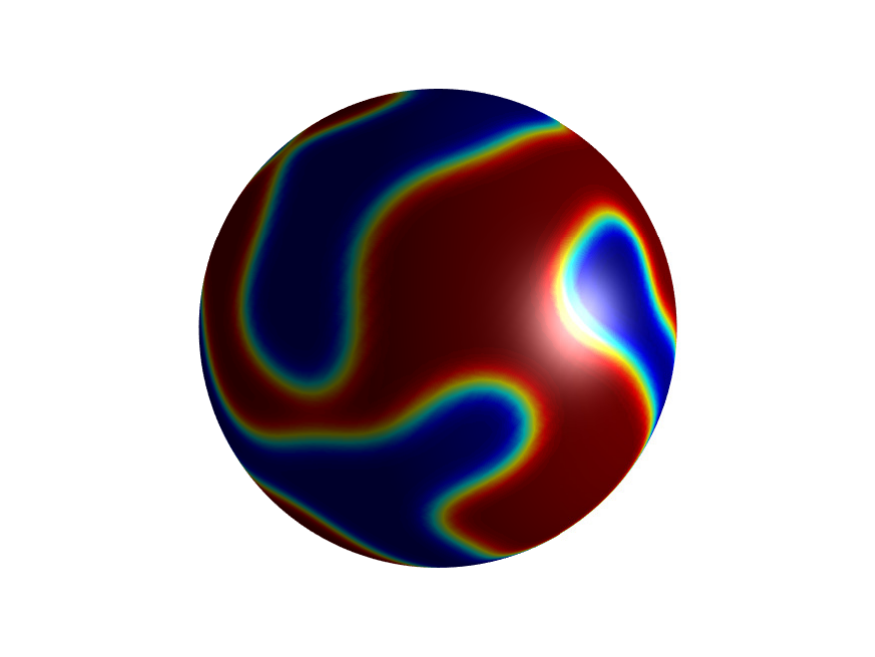} 
\end{overpic}
&
\begin{overpic}[width=0.4\textwidth,trim = 70 50 50 10, clip=true,tics=10]{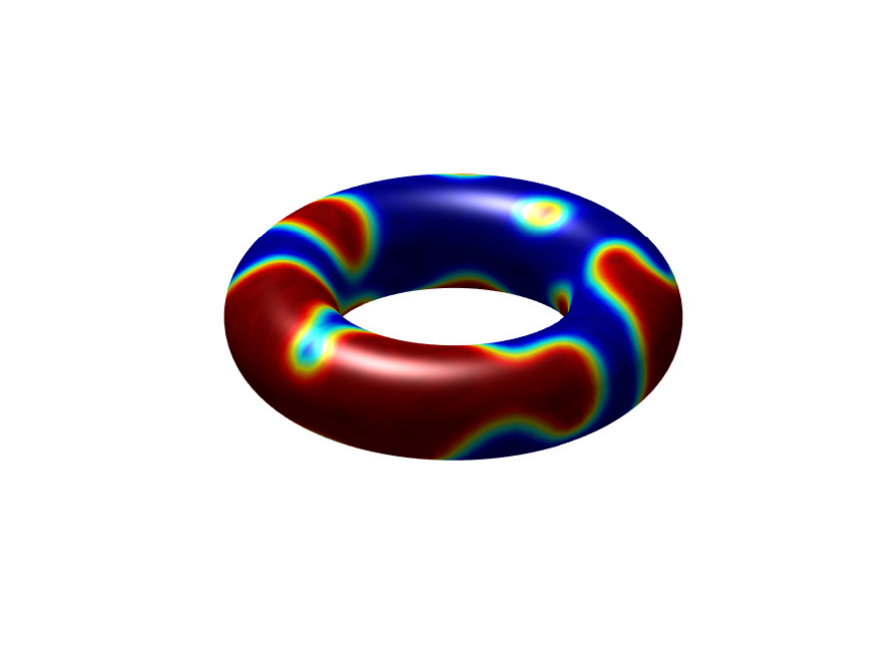} 
\end{overpic}
&\vspace{-0.6cm}\\
\begin{overpic}[width=0.4\textwidth,trim = 10 0 20 5, clip=true,tics=10]{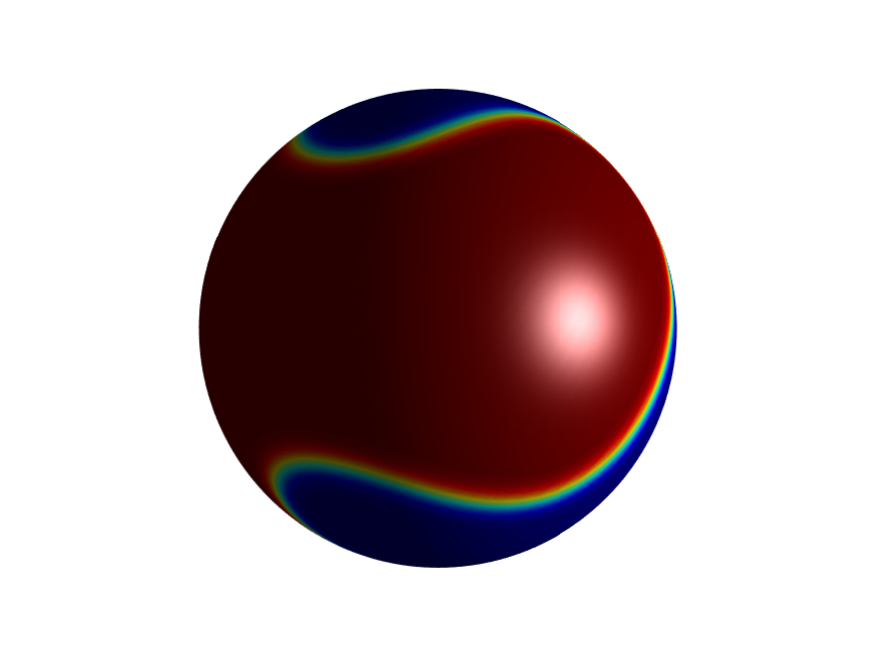} 
\end{overpic}
&
\begin{overpic}[width=0.4\textwidth,trim = 70 50 50 10, clip=true,tics=10]{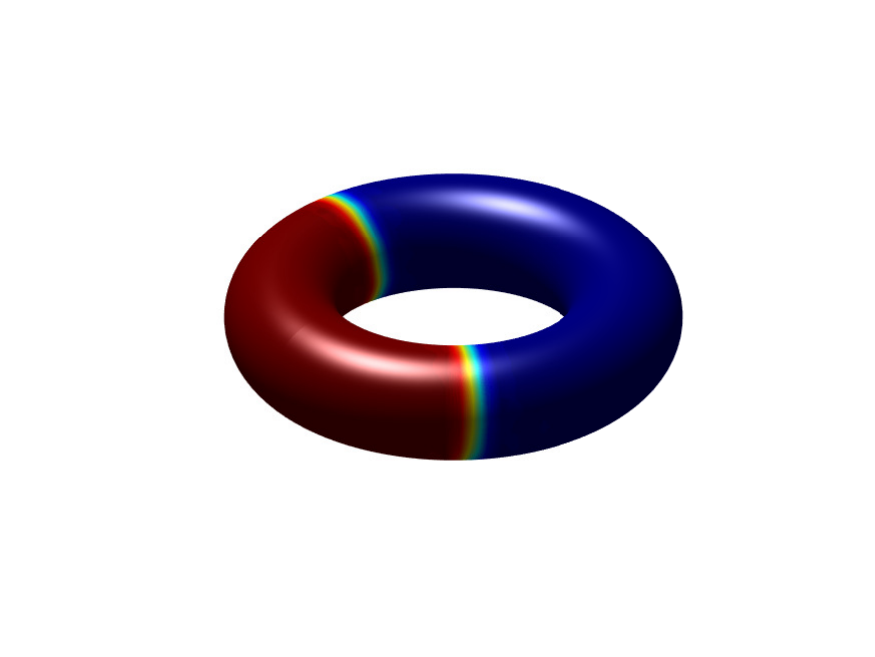} 
\end{overpic}
\end{tabular}
\caption{Phase separation modeled by the Allen-Cahn-equation with the random initial condition \eqref{eq:AC_RandInit} at different times. The left column shows the triangulation on the unit sphere and the numerical solutions at $t=0.01$, $0.03$, $0.3$. The right column shows the triangulation of the torus and corresponding numerical solutions at $t=0.006$, $t=0.02$, $t=0.5$ ($\triangle t=5\times 10^{-4}$, $\varepsilon=0.05$).}\label{fig.AC_randomInit}
\end{figure}

We next explore the phase ordering phenomenon. This phenomenon is commonly observed in diverse fields such as nonequilibrium statistical physics, hydrodynamics theory, and cell biology. For this experiment, the initial state is determined by a random variable uniformly distributed in $[-1,1]^3$ as
\begin{equation}\label{eq:AC_RandInit}
u(x,y,z , 0) \sim \mathcal{U}[-1,1]^3.
\end{equation}
To verify the robustness of the proposed Algorithm~\ref{Algor:1}, we compare it against the meshfree (standard) Galerkin method developed in \cite{kunemund2019high}, which uses thin-plate splines as basis functions and employs the Crank-Nicolson method for time discretization (TPS-CN).

Figure~\ref{fig.AC_randomInit_enerComp} shows the scaled discrete total energy of numerical solutions obtained from Algorithm~\ref{Algor:1} and the method described in \cite{kunemund2019high} for two distinct time step sizes, $\triangle t=2\times 10^{-3}$ and $5\times 10^{-4}$.
The final time is fixed at $T=0.5$, $\varepsilon=0.05$, and the center set size is $N_X=961$. It is observable that the total energy of Algorithm~\ref{Algor:1} remains non-increasing and exhibits little changes for both time step sizes $\triangle t$, thereby corroborating the conclusions in Theorem \ref{thm:FullDiscrete}. Conversely, as depicted in the right image of Figure~\ref{fig.AC_randomInit_enerComp}, the energy derived from the TPS-CN method with $\triangle t=2\times 10^{-3}$ displays oscillatory behavior. This confirms that Algorithm~\ref{Algor:1} is robust because it can conserve the energy dissipation property even for relatively large time step sizes. Moreover, Figure~\ref{fig.AC_randomInit} illustrates the triangulation of the unit sphere and the torus, as well as the evolution of phase separation processes on both surfaces until a steady state is reached using the time step size $\triangle t=5\times 10^{-4}$.

\subsection{Simulations of the Cahn-Hilliard equation}

In the final example, we conduct numerical simulations of the Cahn-Hilliard equation on various surfaces to demonstrate the effectiveness of the proposed method. We first address a benchmark problem on the unit sphere \cite{jeong2020conservative}, employing the following initial condition:
\begin{equation}\label{eq:CHInitCond}
\phi(x,0)=-1+\tanh\Big(\frac{\sin^{-1}(R_1)-\cos^{-1}(x_3)}{\sqrt{2}\varepsilon}\Big)
-\tanh\Big(\frac{\sin^{-1}(R_2)-\cos^{-1}(x_3)}{\sqrt{2}\varepsilon}\Big).
\end{equation}
The initial solution depicts two concentric circles with radii $R_1=0.8$ and $R_2=0.4$. Utilizing Algorithm~\ref{Algor:1} with \eqref{eq:CHfullDiscretization}, we simulate the Cahn-Hilliard equation with $\varepsilon=0.05$ using $N_X=1621$ centers and $N_Y=10001$ quadrature points. Figure~\ref{fig.CH_Init1_sph} shows the numerical solutions at different stages, which displays the gradual shrinking of the two concentric circles until the inner one completely vanishes.

Subsequently, we investigate the phase separation processes with random initial conditions on two distinct surfaces, an
ellipsoid\footnote{\scriptsize $\left\{(x_1,x_2,x_3)\in\R^3\big|x_1^2/a^2+x_2^2/b^2+x_3^2/c^2=1\right\}$, $a=2,b=1,c=1.5$.} and Dupin's cyclide\footnote{\scriptsize $\left\{(x_1,x_2,x_3)\in\R^3\big|(x_1^2+x_2^2+x_3^2+b^2-d^2)^2-4(ax_1+cd)^2-4b^2x_2^2\right\}$, $a=2,b=1.9,d=1,c^2=a^2-b^2$.}.
Figure~\ref{fig.CH_Init1_Surfaces} presents the triangulation on both surfaces and the numerical simulation outcomes at various times with $\triangle t=2\times 10^{-4}$ and $\varepsilon=0.05$. The corresponding scaled discrete energies of the numerical solutions are depicted in Figure~\ref{fig.CH_randomInit_enerComp}. Consistent with our previous observations, the discrete energy remains non-increasing for both surfaces, which further validates the effectiveness of Algorithm~\ref{Algor:1}.

\begin{figure}[htbp]
\begin{tabular}{ccccc}
\begin{overpic}[width=0.165\textwidth,trim= 120 80 120 45, clip=true,tics=10]{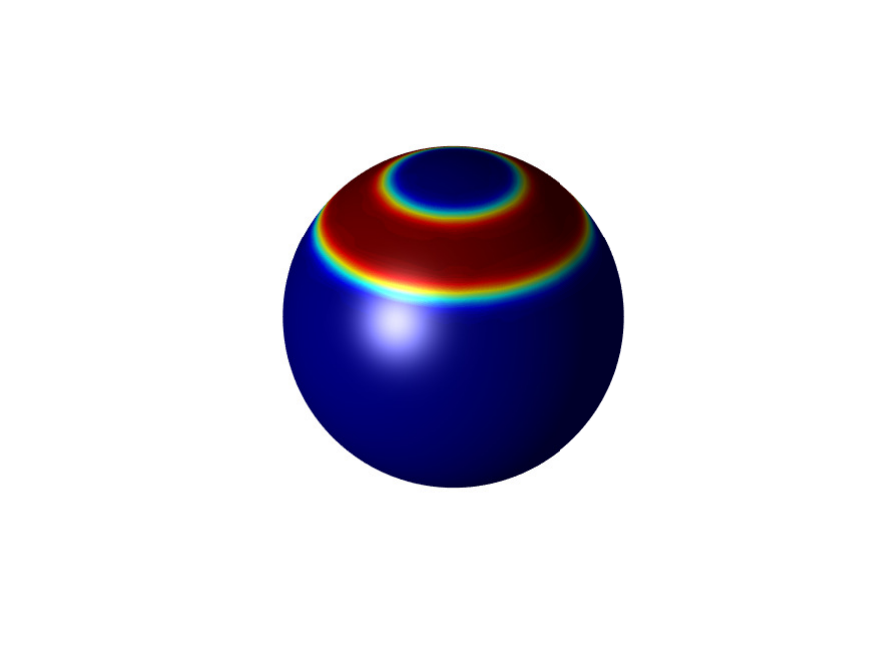}
\put (38,105) {\footnotesize $t=0$}
\end{overpic}
&
\begin{overpic}[width=0.165\textwidth,trim= 120 80 120 55, clip=true,tics=10]{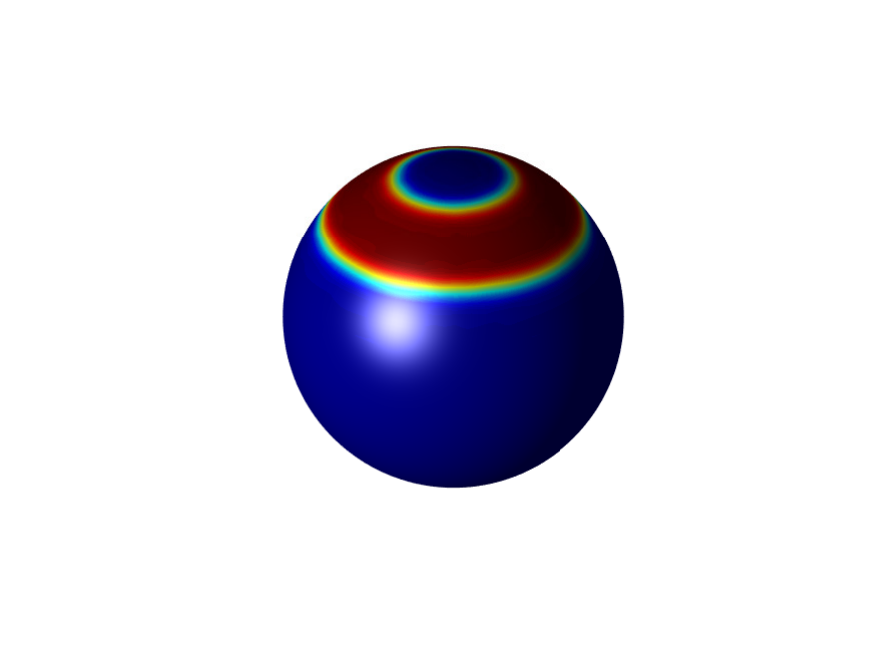}
\put (42,110) {\footnotesize $0.02$}
\end{overpic}
&
\begin{overpic}[width=0.165\textwidth,trim= 120 80 120 55, clip=true,tics=10]{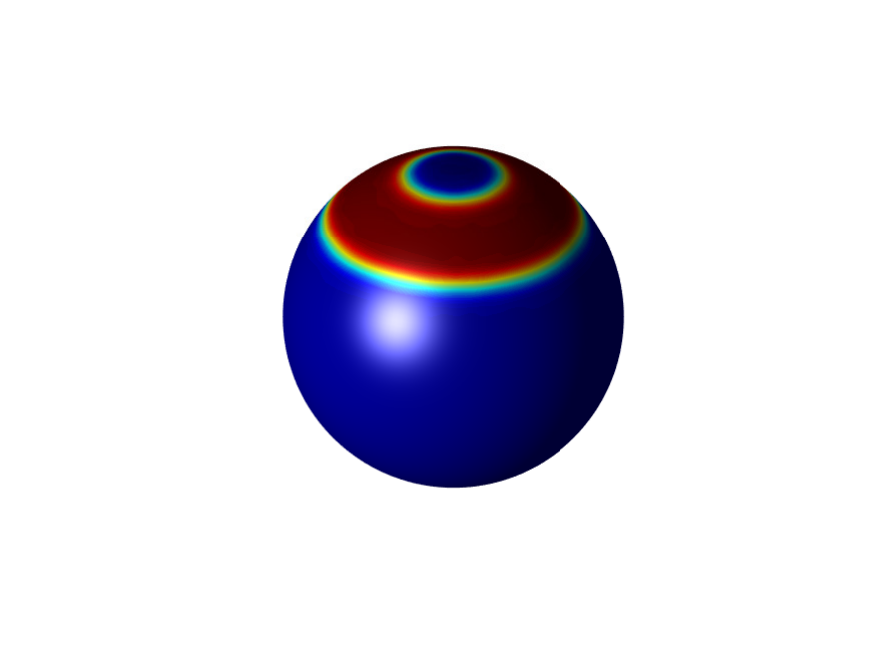} 
\put (42,110) {\footnotesize $0.04$}
\end{overpic}
&
\begin{overpic}[width=0.165\textwidth,trim= 120 80 120 55, clip=true,tics=10]{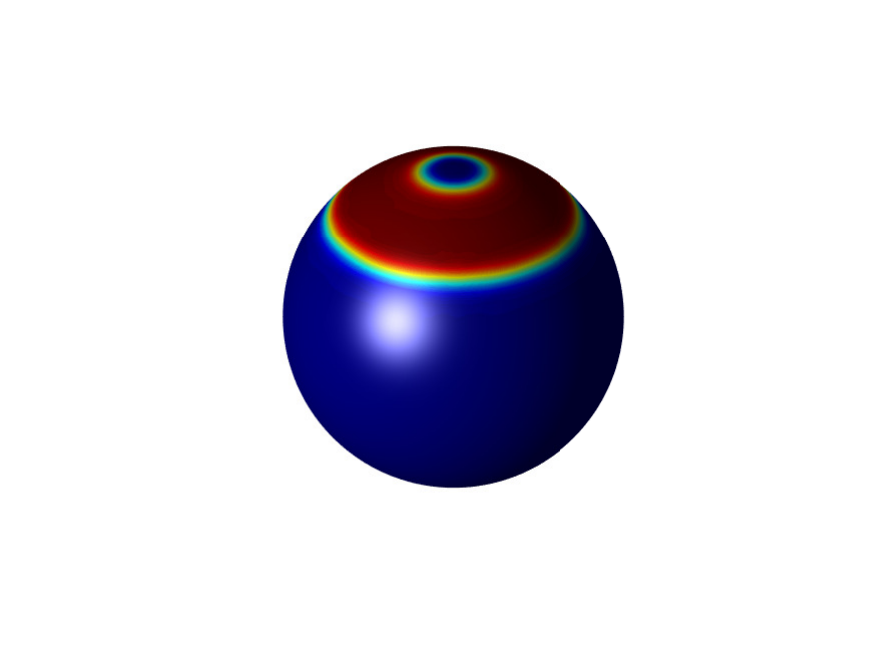} 
\put (42,110) {\footnotesize $0.06$}
\end{overpic}
&
\begin{overpic}[width=0.165\textwidth,trim=120 80 120 55, clip=true,tics=10]{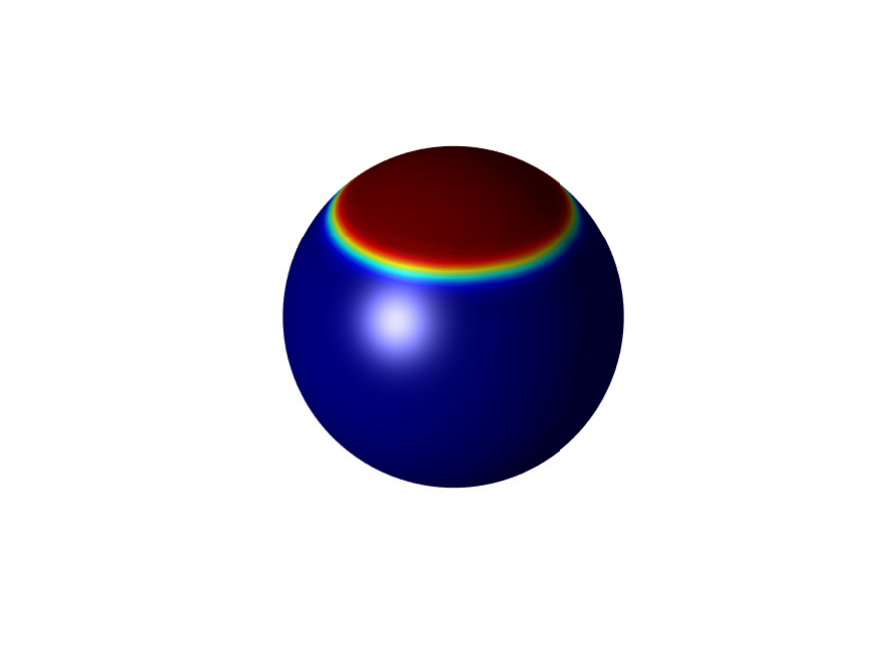} 
\put (42,110) {\footnotesize $0.08$}
\end{overpic}
\end{tabular}
\caption{
Numerical solutions of the Cahn-Hilliard equation on the unit sphere, using the initial condition as defined in \eqref{eq:CHInitCond}. The visualization vividly represents the evolution of the solution over time.
}\label{fig.CH_Init1_sph}
\end{figure}

\begin{figure}[htbp]
\begin{tabular}{ccccc}
\begin{overpic}[width=0.45\textwidth,trim= 80 60 80 45, clip=true,tics=10]{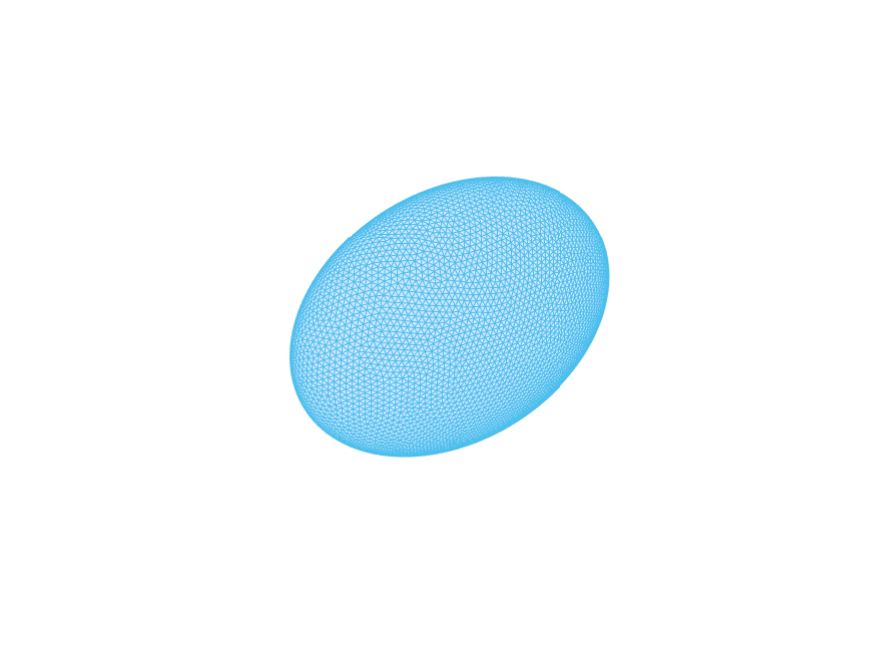}
\end{overpic}
&
\begin{overpic}[width=0.45\textwidth,trim= 80 60 80 45, clip=true,tics=10]{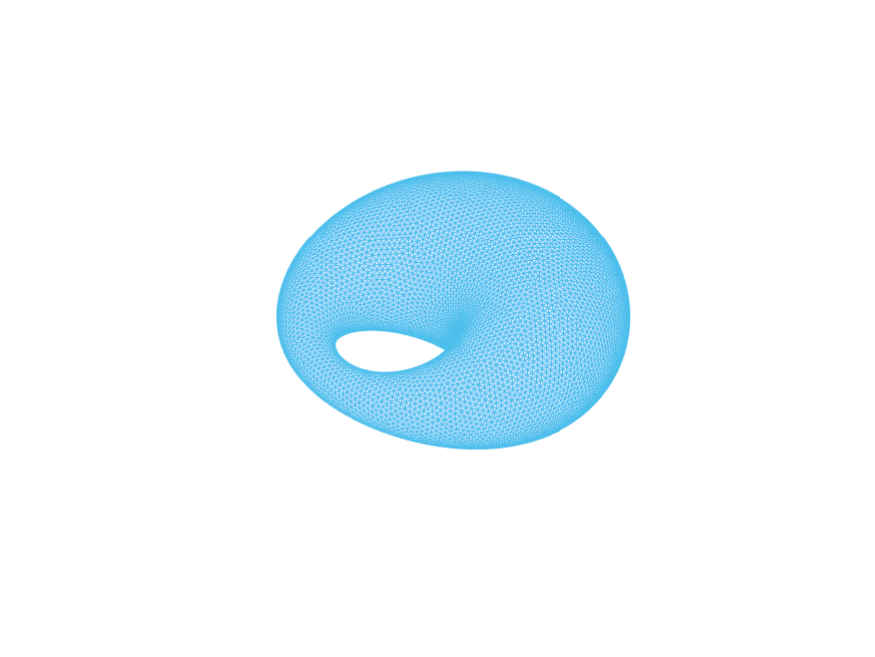}
\end{overpic}
&\\
\begin{overpic}[width=0.45\textwidth,trim = 60 60 60 60, clip=true,tics=10]{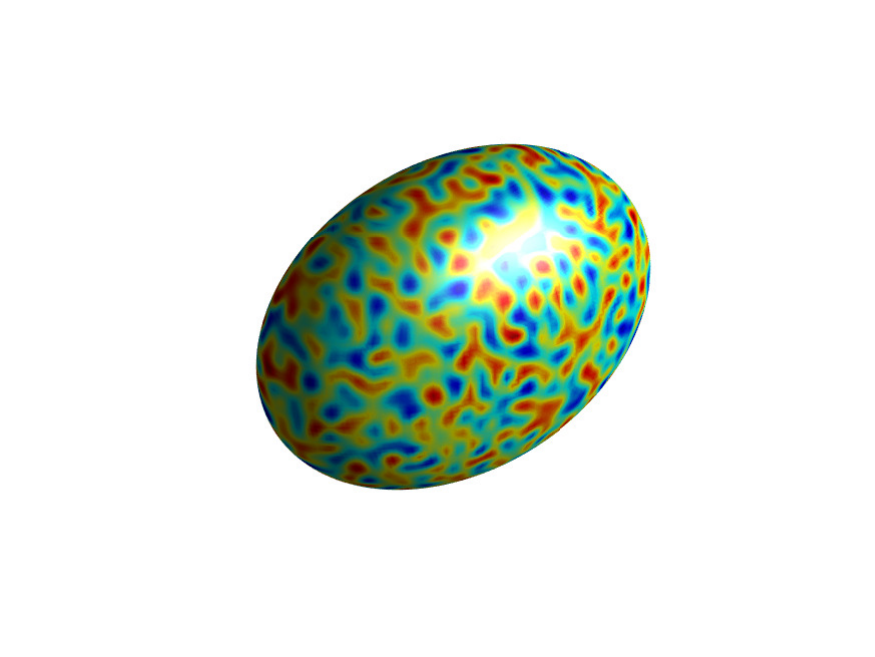} 
\end{overpic}
&
\begin{overpic}[width=0.45\textwidth,trim = 60 60 60 60, clip=true,tics=10]{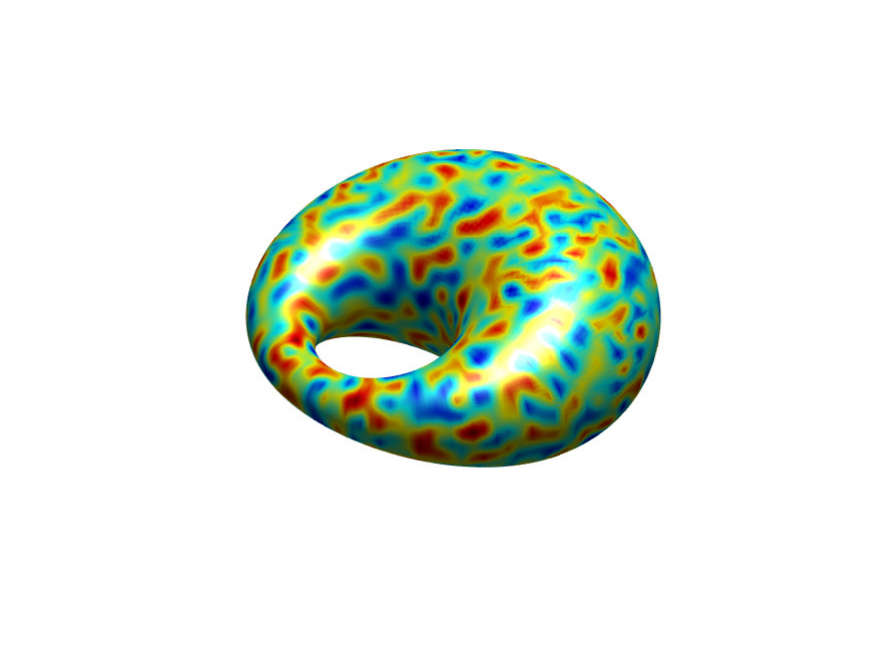} 
\end{overpic}
&\\
\begin{overpic}[width=0.45\textwidth,trim = 60 60 60 60, clip=true,tics=10]{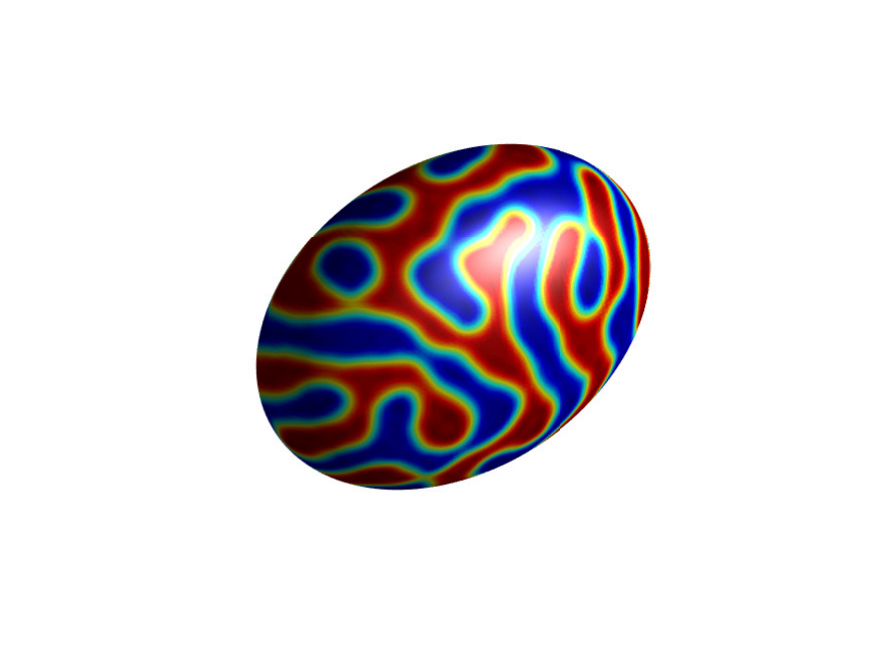} 
\end{overpic}
&
\begin{overpic}[width=0.45\textwidth,trim = 60 60 60 60, clip=true,tics=10]{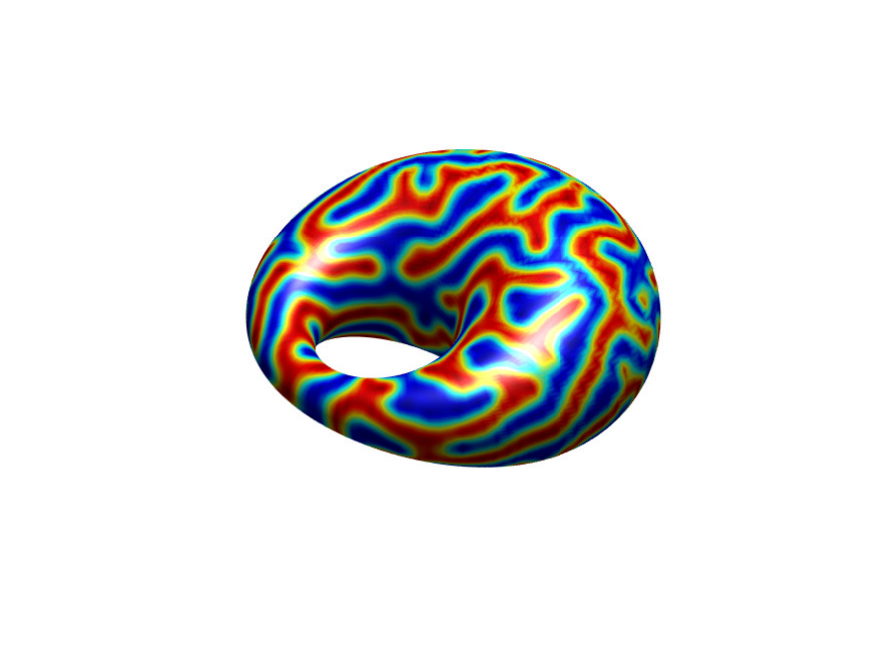} 
\end{overpic}&\\
\begin{overpic}[width=0.45\textwidth,trim = 60 60 60 60, clip=true,tics=10]{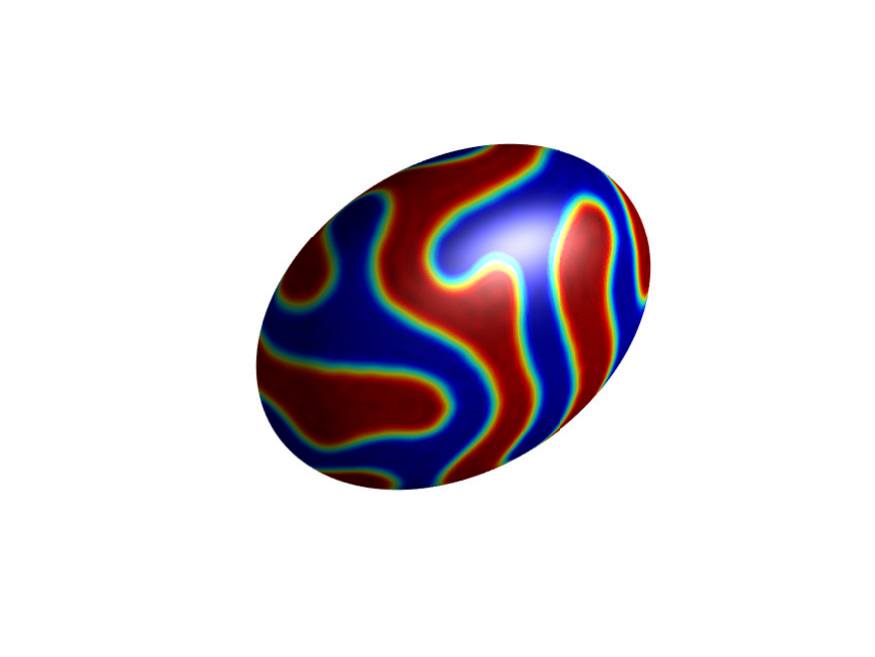} 
\end{overpic}
&
\begin{overpic}[width=0.45\textwidth,trim = 60 60 60 60, clip=true,tics=10]{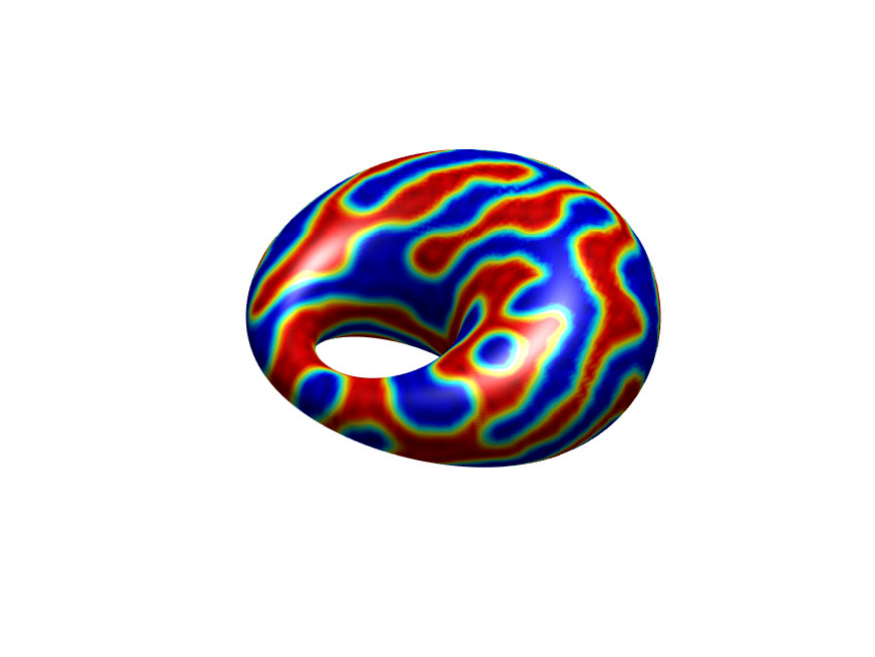} 
\end{overpic}
\end{tabular}
\caption{Phase separation modeled by the Cahn-Hilliard equation with the random initial condition at different times. The left column shows the triangulation on Ellipsoid and the numerical solutions at $t=0$, $0.008$, $0.04$. The right column shows the triangulation of Dupin's cyclide and corresponding numerical solutions ($\triangle t=2\times 10^{-4}$, $\varepsilon=0.05$).}\label{fig.CH_Init1_Surfaces}
\end{figure}

\begin{figure}[htbp]
\centering
\begin{tabular}{ccccc}
\begin{overpic}[width=0.45\textwidth,trim= 0 0 0 0, clip=true,tics=10]{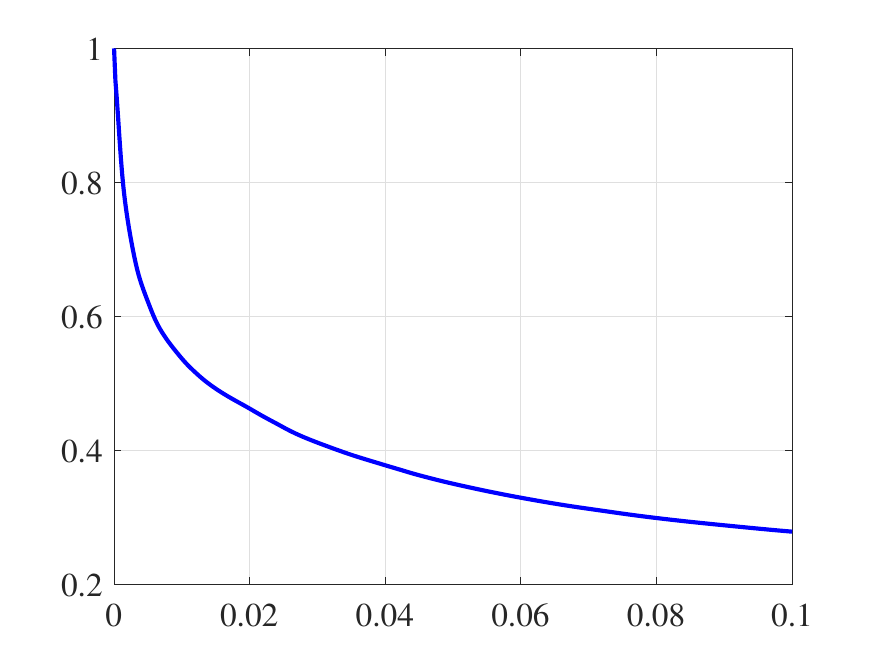}
\put (52,-2) {$t$}
\put (0,32) {\scriptsize \rotatebox{90}{Energy}}
\put (42,75) {\scriptsize Ellipsoid}
\end{overpic}
&
\begin{overpic}[width=0.45\textwidth,trim= 0 0 0 0, clip=true,tics=10]{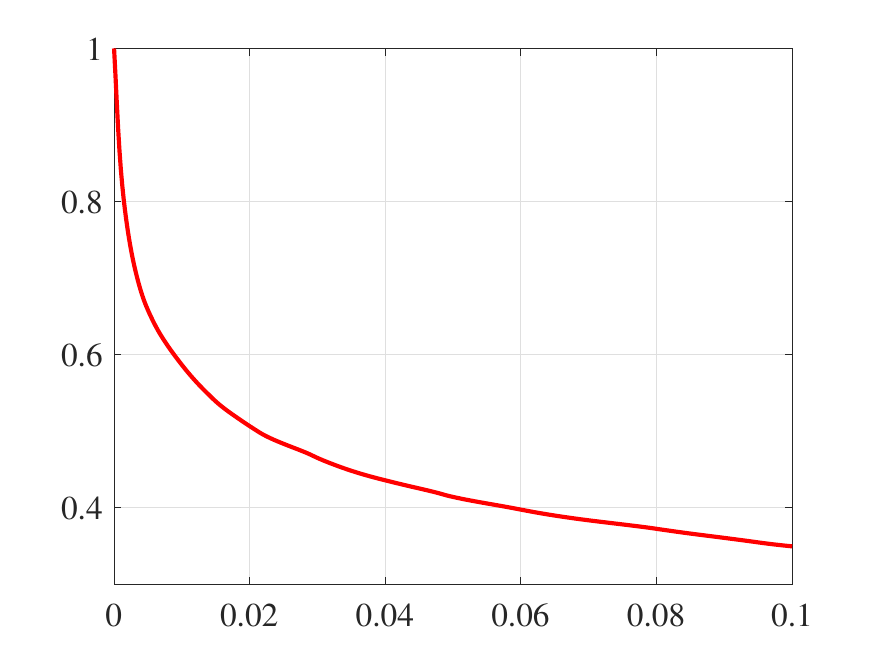}
\put (52,-2) {$t$}
\put (0,32) {\scriptsize \rotatebox{90}{Energy}}
\put (38,75) {\scriptsize Dupin's Cyclide}
\end{overpic}
\end{tabular}
\caption{Scaled energy of the Cahn-Hilliard equation with the random initial condition on the ellipsoid and Dupin's cyclide.}\label{fig.CH_randomInit_enerComp}
\end{figure}

\section{Conclusion}
%

The novelty in this research lies in the introduction of a meshless structure-preserving Galerkin scheme for dissipative partial differential equations on surfaces.
By formulating the PDE model in a variational form, we explore solutions within a finite-dimensional approximation space spanned by positive definite kernels. Our semi-discrete system is shown to preserve the energy dissipation property. Further, we extend the framework to derive a fully-discrete, structure-preserving scheme using the AVF method.

To demonstrate the effectiveness of our proposed approach, we present detailed schemes for solving two specific equations: the Allen-Cahn equation and the Cahn-Hilliard equation. We also delve into the convergence results, specifically for the Allen-Cahn equation.
Additionally, we outline the implementation details of our method, which include the computation of quadrature weights, the mass matrix, and the stiffness matrix. To manifest the accuracy and energy dissipation properties, we conduct several numerical simulations on various surfaces. The results indicate that the meshless Galerkin method employing local Lagrange functions achieves high accuracy while preserving the essential energy dissipation property.
Our future research will aim to apply the developed method to solve other types of structure-preserving PDEs on more complex surfaces.

\section*{\Large Declarations}
\noindent\textbf{Conflicts of interest}~~ ~The authors have no conflicts of interest to declare that are relevant to the content of this
article.\\

\noindent\textbf{Ethical approval}~~ ~The datasets and algorithms generated during the current study are available from the corresponding author on reasonable request.

\bibliographystyle{plain}
\bibliography{meshlessConservative}

\end{document}